\documentclass[twocolumn]{autart_arxiv}    

\usepackage{graphicx}          

\let\theoremstyle\relax

\usepackage[monochrome]{color}
\usepackage{amsthm}
\usepackage{mathtools}
\usepackage{hslTR}
\usepackage{rgsEnvironments}
\usepackage{rgsMacros}
\usepackage{amsmath,amssymb,amsfonts}
\usepackage{cases}
\usepackage{enumitem,kantlipsum}
\usepackage{subfiles}
\usepackage{subfig}
\usepackage{booktabs}
\usepackage{multirow}
\usepackage{siunitx} 
\usepackage{floatrow}
\usepackage{xfrac}
\usepackage[normalem]{ulem}
\usepackage{appendix}

\usepackage{cite}
\usepackage{hyperref}

\usepackage{etoolbox}
\providetoggle{long}
\settoggle{long}{false}

\providetoggle{dissertation}
\settoggle{dissertation}{false}

\providetoggle{arxiv}
\settoggle{arxiv}{true}

\newcommand*\squeezespaces[1]{
  \thickmuskip=\scalemuskip{\thickmuskip}{#1}%
  \medmuskip=\scalemuskip{\medmuskip}{#1}%
  \thinmuskip=\scalemuskip{\thinmuskip}{#1}%
  \nulldelimiterspace=#1\nulldelimiterspace
  \scriptspace=#1\scriptspace
}
\newcommand*\scalemuskip[2]{%
  \muexpr #1*\numexpr\dimexpr#2pt\relax\relax/65536\relax
}

\makeatletter
\renewenvironment{proof}[1][\proofname] {\par\pushQED{\qed}\normalfont\topsep6\p@\@plus6\p@\relax\trivlist\item[\hskip\labelsep\bfseries#1\@addpunct{:}]\ignorespaces}{\popQED\endtrivlist\@endpefalse}
\makeatother

\makeatletter
\DeclareRobustCommand*\cal{\@fontswitch\relax\mathcal}
\makeatother

\begin{document}

\begin{frontmatter}

\title{Robust Parameter Estimation for \\ Hybrid Dynamical Systems} 


\author[UCSC]{Ryan S. Johnson}\ead{rsjohnso@ucsc.edu},    
\author[Mitsubishi]{Stefano Di Cairano}\ead{dicairano@ieee.org},               
\author[UCSC]{Ricardo G. Sanfelice}\ead{ricardo@ucsc.edu}  

\address[UCSC]{University of California, Santa Cruz, CA 95064, USA}  
\address[Mitsubishi]{Mitsubishi Electric Research Laboratories, Cambridge, MA 02139, USA}             


\begin{abstract}                
We consider the problem of estimating a vector of unknown constant parameters for a class of hybrid dynamical systems -- that is, systems whose state variables exhibit both continuous (flow) and discrete (jump) evolution.
Using a hybrid systems framework, we propose a hybrid estimation algorithm that can operate during both flows and jumps that, under a notion of hybrid persistence of excitation, guarantees convergence of the parameter estimate to the true value.
Furthermore, we show that the parameter estimate is input-to-state stable with respect to a class of hybrid disturbances. 
Simulation results including a spacecraft application show the merits of our proposed approach.
\vspace{-1mm}

\end{abstract}
\end{frontmatter}

\section{Introduction}

The estimation of unknown parameters in dynamical systems has been an active research area for years \cite{Tao_adaptive_2003}. 
Parameter estimation algorithms typically rely on exploiting information about the structure of the system along with the available input and output signals to compute online an estimate of the unknown parameters.
One of the most popular estimation problems is recursive linear regression, for which the estimation scheme is often based on the gradient descent algorithm \cite{Narendra_adaptive_1989,Tao_adaptive_2003}.
For dynamical systems, control strategies leveraging estimation algorithms, such as model-reference adaptive control, are used in several engineering applications \cite{DydekQuadrotor2013,DydekNASA2010}.

More recently, there has been a growing interest in hybrid dynamical systems. These systems are characterized by state variables that may evolve continuously (flow) and, at times, evolve discretely (jump) \cite{220}. Hybrid systems provide new and promising modeling frameworks for a wide range of applications including robotics, aerospace, automotive, and power systems \cite{10.1007/3-540-36580-X_33,CRANE201894,Ripaccioli2009Automotive,Kolmanovsky1996Hybrid,DiCairano2014Hybrid,Borrelli2006Traction,DiCairano2013Vehicle}. However, the hybrid nature of these systems limits the applicability of existing continuous-time or discrete-time estimation algorithms. 

Much work has been done on parameter estimation and system identification for specific sub-classes of hybrid systems, such as switched systems \cite{Quandt1958Switched,Ragot2003Switched,Garulli2012Survey} and piecewise-affine systems \cite{Billings1987Affine,Bemporad2005Affine}. 
However, these systems exhibit nonsmooth but continuous evolution of the state variables, rather than jumps in the state variables, hence such results are not applicable to a general class of hybrid systems.
Recently, the work \cite{242} proposed a hybrid estimation algorithm for linear regression with hybrid signals, and 
\cite{Massaroli2020Identification} proposed an algorithm for identification of hybrid systems with linear dynamics. 
However, to the best of our knowledge, before our preliminary results related to this work reported in \cite{255}, an algorithm for estimating unknown parameters for a general class of hybrid systems had not been established in the literature. The goal of this paper is to fill that gap. 
Note that this paper focuses only on online estimation of unknown parameters in the dynamics of hybrid systems. Simultaneous estimation of the flow and jump maps and the flow and jump sets of a hybrid system, i.e., the identification of an entire hybrid system, is still an open problem.

In this paper, we propose a hybrid algorithm for estimating unknown parameters for a class of hybrid systems with nonlinear dynamics that are affine in the unknown parameters. We establish sufficient conditions that guarantee exponential convergence of the parameter estimate to the true value, and we lower bound the convergence rate of the parameter estimate. 
%
%
%
The main contributions in this paper are the following:
\begin{enumerate}[label=\arabic*.]
\item \uline{Estimation under hybrid persistence of excitation}: in the main stability result (Theorem~\ref{thm:PEGlobalStabilityHSg}), we establish that our algorithm guarantees exponential convergence of the parameter estimate to the true value under a notion of hybrid persistence of excitation (PE) inspired by \cite{242}. 
To the best of our knowledge, it is the first hybrid PE sufficient condition that ensures estimation of parameters for the considered class of hybrid systems. \\
%
\item \uline{Estimator robustness to hybrid disturbances}: to prove the main robustness result (Theorem~\ref{thm:PEHgISS}), we generalize the error dynamics of our algorithm to a class of hybrid systems, denoted by ${\cal H}$, that includes hybrid disturbances.
Lemma~\ref{lem:PEP} and Theorem~\ref{thm:COMISS} construct an input-to-state stability (ISS) Lyapunov function for ${\cal H}$
by extending ISS results for continuous-time and discrete-time systems \cite{Khalil2002Nonlinear,Bof2018Lyapunov,Sontag1995ISS,Sontag2008ISS,Jiang2001ISS,Cai2009ISS}.
%

%
\end{enumerate}

In \cite{255}, we proposed a hybrid algorithm for estimating unknown parameters in a class of hybrid systems with linear dynamics. We showed that the parameter estimate converges to the true value if the hybrid regressor satisfies the classical continuous-time PE condition during flows and the classical discrete-time PE condition at jumps.
In comparison to \cite{255}, this paper considers a wider class of hybrid systems with nonlinear dynamics and proposes a new hybrid algorithm to solve the estimation problem. Moreover, here we relax the classical PE conditions imposed in \cite{255}, and instead impose a hybrid PE condition inspired by \cite{242}.
%

%
%
The hybrid PE condition is exploited in more recent work involving authors of this paper and of \cite{242} to establish uniform exponential stability for a general class of time-varying hybrid dynamical systems \cite{SaoudHybridPE2023}.
In comparison to \cite{SaoudHybridPE2023}, this paper focuses on deriving and analyzing the properties induced by an algorithm for estimating unknown parameters in a class of hybrid dynamical systems when the state of the hybrid plant is measured, while \cite{SaoudHybridPE2023} allows for only output measurements.
Due to such differences and the different techniques used, we provide explicit bounds on the convergence rate of the parameter estimate, and on the estimation error when noise is present whereas \cite{SaoudHybridPE2023} establishes only the existence of such bounds.
Our analysis does not impose completeness of maximal solutions, while the results in \cite{SaoudHybridPE2023} rely on completeness of maximal solutions to ensure well-posedness.
The approach to the proof of exponential stability in this paper differs from that in \cite{SaoudHybridPE2023}, specifically, we leverage a property of input-to-state stability with respect to an exponentially convergent hybrid signal, while the analysis in \cite{SaoudHybridPE2023} relies on a general, but more abstract, result involving uniform observability properties.
%
%

The remainder of this paper is organized as follows. Preliminaries on continuous-time and discrete-time estimation algorithms are presented in Section~\ref{sec:PEPreliminaries}. In Section~\ref{sec:PEMotivation}, we present a motivational example that highlights the limitations of these algorithms, that our hybrid algorithm aims at overcoming. Our algorithm is described in Section~\ref{sec:PEProblemStatement}, and the stability and robustness properties are analyzed in Section~\ref{sec:PEStabilityAnalysis} and Section~\ref{sec:PERobustnessAnalysis}, respectively. Simulation results are in Section~\ref{sec:PESimulationResults}. Conclusions and future work are in Section~\ref{sec:PEConclusion}. 
\iftoggle{arxiv}{For readability, the proofs of Theorem~\ref{thm:COMISS}, Theorem~\ref{thm:COMISS2}, Lemma~\ref{prop:PEepsStab}, Lemma~\ref{prop:PEpsiStab}, and Lemma~\ref{lem:PEepsISS} are in the Appendix.}
{The proofs of Theorems~\ref{thm:COMISS} and~\ref{thm:COMISS2} are in the Appendix. Due to space constraints, the proofs of Lemmas~\ref{prop:PEepsStab},~\ref{prop:PEpsiStab}, and~\ref{lem:PEepsISS} are in \cite{JohnsonHybridPE2023}.}

\section{Preliminaries}
\label{sec:PEPreliminaries}
\subsection{Notation}
%
We denote the set of real, nonnegative real, and positive real numbers by $\reals$, $\realsgeq$, and $\realsg$, respectively. We denote the set of natural numbers (including zero) as $\nats$. The matrix $I$ denotes the identity matrix of appropriate dimension. 
The Euclidean norm of vectors and the associated induced matrix norm are denoted by $|\cdot|$, and the Frobenius norm is denoted by $|\cdot|_{\mathrm F}$.
%
%
Given a matrix $A \in \reals^{n \times n}$, $\text{\rm eig} (A)$ denotes the set of all eigenvalues of $A$, $\lambda_{\min}(A) := \min\{ \lambda/2 : \lambda \in \text{\rm eig} (A + A^\top) \}$, and $\lambda_{\max}(A) := \max\{ \lambda/2 : \lambda \in \text{\rm eig} (A + A^\top) \}$.
For $x,y \in \reals^n$, we write $[x^\top \ \, y^\top]^\top$ as $(x,y)$.
The distance of a point $x$ to a nonempty set $S$ is $|x|_S = \inf_{y \in S} |y-x|$ (the quantity $|\cdot|_S$ should not be confused with the Frobenius norm $|\cdot|_{\mathrm F}$ since, besides the subscript being $S$ in the former, the argument of the former is a vector and the argument of the latter is a matrix).
Given a set-valued mapping $M : \reals^m \rightrightarrows \reals^n$, the domain of $M$ is $\dom M = \{ x \in \reals^m \ : \ M(x) \neq \varnothing \}$ and the range of $M$ is $\rge M = \{ y \in \reals^n : \exists \, x \in \reals^m, y \in M(x)\}$.
Given sets $S, U \subset \reals^n$, $\cl (S)$ denotes the closure of $S$, and $S \setminus U$ denotes set subtraction.
%
%
%
%
Given a measure space $M$ and a function $f : M \to \reals$, the essential supremum of $f$ is $\esssup_{m \in M} f(m) = \inf_{c \in \reals} \{ |f(m)| \leq c \ \text{\rm for almost all} \ m \in M \}$.
A function $f$ is ${\cal L}_\infty$ ($f \in {\cal L}_\infty$) if $\esssup_{m \in M} f(m)$ is finite.
%
Given $x \in \reals$, the exponential is $\mathrm{e}^{x}$ or $\exp (x)$, equivalently.
%
%
\subsection{Review of Parameter Estimation Algorithms} \label{sec:PEContDiscGrad}
In preparation for our proposed hybrid parameter estimation algorithm, we review relevant continuous-time and discrete-time estimation algorithms.
\begin{itemize}[listparindent=1.5em]
\item Consider the continuous-time system
\begin{align} \label{eqn:PExdot}
    \dot x &= f_c(x,u(t)) + \phi_c(t) \theta 
    \quad \forall t \geq 0,
\end{align}
where $x \in \mathbb{R}^n$ is the known state vector, $t \mapsto u(t) \in \reals^m$ is the known input, $t \mapsto \phi_c(t) \in \mathbb{R}^{n \times p}$ is the known regressor, $(x,u) \mapsto f_c(x,u) \in \reals^n$ is a known continuous function, $\theta \in \mathbb{R}^{p}$ is a vector of unknown constant parameters, and $n, p, m \in \nats$. \newline

\noindent To estimate $\theta$, we convert \eqref{eqn:PExdot} into a form similar to a linear regression model by introducing \cite{Panteley2002Detectability} the state variables $\psi \in \reals^{n \times p}$ and $\eta \in \reals^{n}$ with dynamics
\begin{equation} \label{eqn:PEtemp1}
\begin{aligned}
    \dot \psi &= -\lambda_c \psi + \phi_c(t) \\
    \dot \eta &= -\lambda_c (x + \eta) - f_c(x,u(t)),
\end{aligned}
\end{equation}
where $\lambda_c > 0$ is a design parameter. 
Defining
%
    $\varepsilon := x + \eta - \psi \theta$ 
%
and $y := x + \eta$, it follows that $\psi$ and $\varepsilon$ are related via
%
    $y = \psi \theta + \varepsilon$. 
%
Since $\theta$ is constant, differentiating $\varepsilon$ along trajectories of \eqref{eqn:PExdot}, \eqref{eqn:PEtemp1} yields $\dot \varepsilon = - \lambda_c \varepsilon$. 
%
%
Thus, $\varepsilon$ converges exponentially to zero. 
Moreover, we have the following equivalences: 
%
    $\varepsilon~\rightarrow~0
    \iff x + \eta~\rightarrow~\psi \theta
    \iff y~\rightarrow~\psi \theta.$
%
Hence, $y$, $\psi$, and $\theta$ are related via a linear regression model plus an exponentially convergent term. \newline 

\noindent Denoting the estimate of $\theta$ as $\hat \theta$, the gradient algorithm for $\hat \theta$ is \cite{Narendra_adaptive_1989} 
\begin{align} \label{eqn:PEhatthetaC}
    \dot{\hat \theta} &= \gamma_c \psi^\top (y - \psi \hat \theta),
\end{align}
where $\gamma_c > 0$ is a design parameter. \newline
\item Consider the discrete-time system
\begin{align*} 
    x(j+1) &= g_d(x(j),u(j)) + \phi_d(j) \theta 
    \quad \forall j \in \nats,
\end{align*}
%
where $x \in \mathbb{R}^n$ is the known state vector, $j \mapsto u(j) \in \reals^m$ is the known input, $j \mapsto \phi_d(j) \in \mathbb{R}^{n \times p}$ is the known regressor, $(x,u) \mapsto g_d(x,u) \in \reals^n$ is a known continuous function, $\theta \in \mathbb{R}^p$ is a vector of unknown constant parameters, and $n, p, m \in \nats$. \newline

\noindent Using similar reasoning as in the continuous-time case, we estimate $\theta$ using a gradient algorithm as
%
%
\begin{align} \label{eqn:PEhatthetaD}
    \hat \theta(j+1) &= \hat \theta(j) \\
    & \mbox{$\squeezespaces{0.5}\displaystyle+ \frac{\psi(j+1)^\top}{\gamma_d + |\psi(j+1)|^2} ( y(j+1)
    - \psi(j+1) \hat \theta(j) ) ,$} \notag
\end{align}
%
with
\begin{equation} \label{eqn:PEpsietaplus}
\hspace{-5.5mm}
\mbox{$\squeezespaces{0.5}
\begin{aligned}
    \psi(j+1) &= (1 - \lambda_d) \psi(j) + \phi_d(j) \\
    \eta(j+1) &= (1 - \lambda_d) (x(j) + \eta(j)) - g_d(x(j),u(j)) \\
    y(j+1) &= x(j+1) + \eta(j+1),
\end{aligned}
$}
\hspace{-4mm}
\end{equation}
where $\gamma_d > 0$, $\lambda_d \in (0, 2)$ are design parameters.\newline
%

\noindent To compute the update law for $\hat \theta$ in \eqref{eqn:PEhatthetaD}, we require measurements of $x$ for two consecutive discrete time steps. 
Moreover, two computational steps are required to update $\hat \theta$ at time $j \in \nats$. The first step computes $\psi(j+1)$, $\eta(j+1)$, and $y(j+1)$ in \eqref{eqn:PEpsietaplus}, and the second step computes $\hat \theta(j+1)$ in \eqref{eqn:PEhatthetaD}. For simplicity, we omit the first computational step in \eqref{eqn:PEhatthetaD}. 
%
%
\end{itemize}

It is shown in \cite{Narendra_adaptive_1989} that, if $\varepsilon = 0$, the following PE condition is necessary and sufficient for convergence of $\hat \theta$ in \eqref{eqn:PEhatthetaC} to $\theta$:
\begin{enumerate}[label=(C1),leftmargin=*,labelindent=0.5em]
    \item \label{C1} The signal $t \mapsto \psi(t)$ is uniformly bounded and there exist $T > 0$ and $\mu > 0$ such that 
    \begin{align} \label{eqn:COMPEContinuous}
        \int_{t}^{t + T} \psi(s)^\top \psi(s) ds \geq \mu I
        \quad \forall t \geq 0.
    \end{align}
\end{enumerate}
For the discrete-time case, the PE condition is \cite{Tao_adaptive_2003}:  
\begin{enumerate}[label=(C2),leftmargin=*,labelindent=0.5em]
    \item \label{C2} The signal $j \mapsto \psi(j)$ is uniformly bounded and there exist $J \in \nats \setminus \{0\}$ and $\mu > 0$ such that 
    \begin{align} \label{eqn:COMPEDiscrete}
        \sum_{i = j}^{j + J} \psi(i)^\top \psi(i) \geq \mu I
        \quad \forall j \in \nats.
    \end{align}
\end{enumerate}

\subsection{Hybrid Dynamical Systems} \label{sec:COMHybridSystems} 
In this paper, a hybrid system $\mathcal{H}$ is defined by $(C,F,D,G)$ as \cite{220}
\begin{equation} \label{eq:HyEq}
{\cal H} : 
\begin{dcases}
\begin{aligned}
    \dot \xi &= F(\xi) &\qquad  \xi \in C \\
    \xi^+ &= G(\xi) &\qquad \xi \in D,
\end{aligned}
\end{dcases}
\end{equation}
where $\xi \in \reals^n$ is the state, $F : C \to \reals^n$ is the flow map defining the continuous dynamics, and $C \subset \reals^n$ defines the flow set on which flow is permitted. The mapping $G  : D \to \reals^n$ is the jump map defining the law resetting $\xi$ at jumps, and $D \subset \reals^n$ is the jump set on which jumps are permitted.

A solution $\xi$ to $\HS$ is a {\it hybrid arc} \cite{220} that is parameterized by $(t,j) \in \realsgeq \times \nats$, where $t$ is the elapsed ordinary time and $j$ is the number of jumps that have occurred. The domain of $\xi$, denoted $\dom \xi \subset \realsgeq \times \nats$, is a {\it hybrid time domain}, in the sense that for every $(T,J) \in \dom \xi$, there exists a nondecreasing sequence $\{t_j\}_{j=0}^{J+1}$ with $t_0 = 0$ such that
%
$\dom \xi \cap \bigpar{[0,T] \times \{ 0, 1, \dots, J \}} = \bigcup_{j = 0}^J \bigpar{[t_j,t_{j+1}] \times \{j\}}$.
%
A solution $\xi$ to $\HS$ is said to be
\begin{itemize}
    \item {\it nontrivial} if $\dom \xi$ contains more than one point;
    \item {\it continuous} if nontrivial and $\dom \xi \subset \realsgeq \times \{0\}$;
    \item {\it discrete} if nontrivial and $\dom \xi \subset \{0\} \times \nats$.
\end{itemize}
A solution $\xi$ to $\HS$ is called {\it maximal} if it cannot be extended -- that is, if there does not exist another solution $\xi'$ to $\HS$ such that $\dom \xi$ is a proper subset of $\dom \xi'$ and $\xi(t,j) = \xi'(t,j)$ for all $(t,j) \in \dom \xi$. A solution is called {\it complete} if its domain is unbounded.
The operations $\sup_t \dom \xi$ and $\sup_j \dom \xi$ return the supremum of the $t$ and $j$ coordinates, respectively, of points in $\dom \xi$. The length of $\dom \xi$ is $\sup_t \dom \xi + \sup_j \dom \xi$.
%

\iftoggle{arxiv}{\pagebreak}{\pagebreak}We employ the following notion of stability \cite{220}.
%
%
%
%
%
\begin{definition} \label{def:COMSemiGlobExpStab}
Given a hybrid system $\HS$ with data as in \eqref{eq:HyEq}, a nonempty closed set ${\cal A} \subset \reals^n$ is said to be {\it semiglobally pre-exponentially stable}\footnote{The term ``pre-exponential,'' as opposed to ``exponential,'' indicates the possibility of a maximal solution that is not complete. This allows for separating the conditions for completeness from the conditions for stability and attractivity.} for $\HS$ if, for each compact ${\cal X}_0 \subset \reals^n$, there exist $\kappa, \lambda > 0$ such that each solution $\xi$ to $\HS$ from $\xi(0,0) \in {\cal X}_0$ satisfies
\begin{equation} \label{eqn:PESemiGlobExpStability}
    |\xi(t,j)|_{\cal A} \leq \kappa \mathrm{e}^{-\lambda(t + j)} |\xi(0,0)|_{\cal A}
    \ \ \forall (t,j) \! \in \! \dom \xi.
\end{equation}
If there exist $\kappa, \lambda > 0$ such that each solution $\xi$ to $\HS$ satisfies \eqref{eqn:PESemiGlobExpStability}, then ${\cal A}$ is said to be {\it globally pre-exponentially stable} for $\HS$.
\end{definition}
Given a hybrid arc $(t,j) \mapsto \xi(t,j) \in \reals^n$, we denote the supremum norm of $\xi$ from $(0,0)$ to $(t,j)$ as
\begin{equation*} 
\|\xi\|_{(t,j)} := \max \left\{
\mbox{\squeezespaces{0.5}\scriptsize$\displaystyle
\begin{aligned}
    \underset{\substack{ (s,k) \, \in \, \dom \xi \setminus \Upsilon(\dom \xi), \\ (s,k) \, \leq \, (t,j)}}{\esssup} 
    \hspace{-5mm} |\xi(s,k)|,
    \underset{\substack{(s,k) \, \in \, \Upsilon(\dom \xi), \\ (s,k) \, \leq \, (t,j)}}{\sup} 
    \hspace{-3mm} |\xi(s,k)|
\end{aligned}
$}
\right\}
\end{equation*}
where 
\begin{equation} \label{eqn:COMUpsilon}
    \Upsilon(\dom \xi) := \{ (t,j) \in \dom \xi : (t,j+1) \in \dom \xi \}.
\end{equation}

\section{Motivational Example}
\label{sec:PEMotivation}
To motivate our parameter estimation algorithm, consider the hybrid arcs $\phi_c, \phi_d : E \to \reals^{2 \times 2}$ with hybrid time domains
%
    $E = \bigcup_{k = 0}^{\infty} \big( [2 \pi k, \ \pi (2 k + 2)] \times \{k\} \big)$. 
%
%
The values of $\phi_c$ and $\phi_d$ are
$\phi_c(t,j) = \smallmatt{\sin(t) & 0 \\ 0 & 0}$ and
$\phi_d(t,j) = \smallmatt{1 & 2 \\ 2 & 4}$ 
%
%
for all $(t,j) \in E$. 
For such $\phi_c$ and $\phi_d$, consider a hybrid system as in \eqref{eq:HyEq} with an added input\footnote{See \cite{220} for details on hybrid systems with inputs.} $u : E \to \reals$, state $x = (x_1,x_2) \in \reals^2$, and dynamics
\begin{equation} \label{eqn:PEmotivationalplant}
\begin{aligned} 
    \dot x &= \phi_c(t,j) \theta & \quad
    (x,u(t,j)) &\in C_P  \\
    x^+    &= \phi_d(t,j) \theta & \quad
    (x,u(t,j)) &\in D_P,
\end{aligned}
\end{equation}
where $\theta = [1 \, \; 1]^\top$ is a vector of unknown parameters.
The flow and jump sets are
$C_P = ( \reals^2 \times \reals ) \setminus D_p$ and $D_P = \{ (x,u) \in \reals^2 \times \reals : u \geq 2 \pi \}$, respectively.
%
%
The input $u(t,j) = t - 2 \pi j$ for all $(t,j) \in E$ 
is a sawtooth function that periodically ramps to a value of $2 \pi$ and then resets to zero.\footnote{With $C_P$, $D_P$, and $u$ given below \eqref{eqn:PEmotivationalplant}, the hybrid time domain of each maximal solution to the hybrid system in \eqref{eqn:PEmotivationalplant} is equal to the hybrid time domain $E$ of $\phi_c$ and $\phi_d$.}
%
%

%
Given a solution $x$ to \eqref{eqn:PEmotivationalplant} from $x(0,0) = (3,6)$, we want to estimate $\theta$.
To do so, we first separately analyze the flows and jumps of these signals. We define the continuous-time signals $t \mapsto \bar x_c(t) := \smallmatt{4 - \cos(t) \\ 6}$ and $t \mapsto \bar \phi_c(t) := \smallmatt{\sin(t) & 0 \\ 0 & 0}$, which are obtained by neglecting the resets of $x$ and $\phi_c$, respectively, at jumps.
The signals $\bar x_c$ and $\bar \phi_c$ are solutions to the continuous-time system
%
    $\dot{\bar x}_c(t) = \bar \phi_c(t) \theta$
%
for all $t \geq 0$. 
Next, we define the discrete-time signals $j \mapsto \bar x_d(j) := \smallmatt{3 \\ 6}$ and $j \mapsto \bar \phi_d(j) := \smallmatt{1 & 2 \\ 2 & 4}$, which are obtained by neglecting the evolution of $x$ and $\phi_d$, respectively, during flows.
The signals $\bar x_d$ and $\bar \phi_d$ are solutions to the discrete-time system
%
    ${\bar x}_d(j+1) = \bar \phi_d(j) \theta$
%
for all $j \in \nats$. 
Using the transformations in Section \ref{sec:PEContDiscGrad}, we employ the continuous-time and discrete-time algorithms \eqref{eqn:PEhatthetaC} and \eqref{eqn:PEhatthetaD} to estimate $\theta$ in \eqref{eqn:PEmotivationalplant}. The parameter estimation error for both algorithms fails to converge to zero, as shown in Figure~\ref{fig:PEMotovEst}.\footnote{Code at \href{https://github.com/HybridSystemsLab/HybridGD_Motivation}{\url{https://github.com/HybridSystemsLab/HybridGD_Motivation}}} 

To see why the continuous-time algorithm fails to estimate $\theta$, note that for $\bar \phi_c$, the value of $t \mapsto \psi(t)$ in \eqref{eqn:PEtemp1} is
%
    $\psi(t) = \mathrm{e}^{-\lambda_c t} \psi(0) + \int_{0}^{t} \mathrm{e}^{-\lambda_c(t - s)} \bar \phi_c(s) ds$
%
for all $t \geq 0$. 
Since $\mathrm{e}^{-\lambda_c t} \psi(0)$ converges exponentially to zero and the second column of $\bar \phi_c(t)$ is zero for all $t \geq 0$, $t \mapsto \psi(t)$ does not satisfy \ref{C1} for any $T > 0$.
Similarly for $\bar \phi_d$, the value of $j \mapsto \psi(j)$ in \eqref{eqn:PEpsietaplus} is
%
%
    $\psi(j) = (1 - \lambda_d)^j \psi(0) + \sum_{i = 0}^{j-1} (1 - \lambda_d)^{(j-i-1)} \bar \phi_d(i)$
%
for all $j \in \nats$.
Since $(1 - \lambda_d)^j \psi(0)$ converges exponentially to zero and $\bar \phi_d(j)$ is constant and singular for all $j \in \nats$, $j \mapsto \psi(j)$ does not satisfy \ref{C2} for any $J \in \nats \setminus \{0\}$.
%
%
%
On the other hand, the hybrid algorithm proposed in this paper successfully estimates $\theta$ by leveraging the information available during both flows and jumps, as shown in Figure~\ref{fig:PEMotovEst}.
\begin{figure}[!hbt]
    \centering
        \includegraphics[width=1\linewidth,height=\textheight,keepaspectratio]{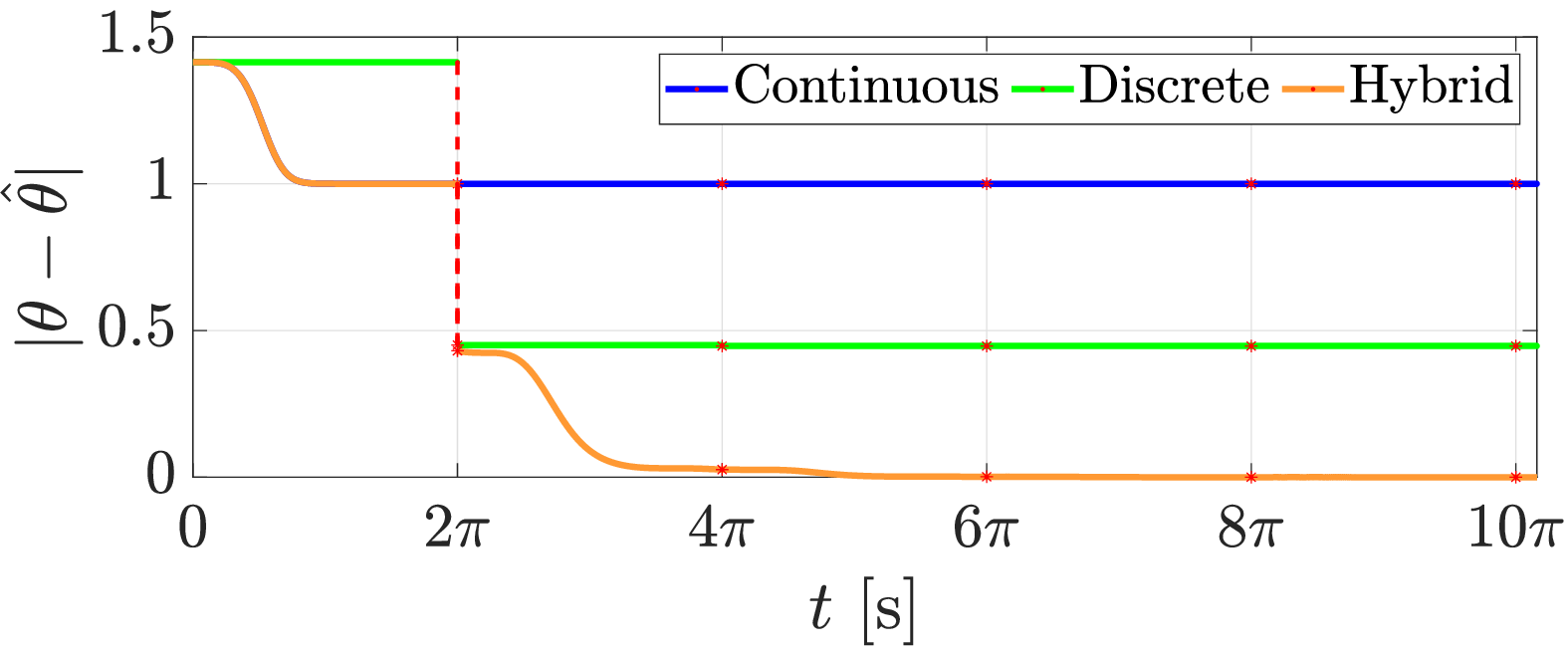}
        \caption{The projection onto $t$ of the norm of the parameter estimation error for the continuous-time and discrete-time estimation algorithms, and our hybrid algorithm. The continuous-time and discrete-time algorithms produce nonzero steady-state error, whereas the error for our algorithm converges to zero.}
    \label{fig:PEMotovEst}
\end{figure}

\section{A Hybrid Parameter Estimation Algorithm} 
\label{sec:PEProblemStatement}
\subsection{Problem Statement}
Motivated by the limitations of the continuous-time and discrete-time estimation algorithms highlighted in Section~\ref{sec:PEMotivation}, we develop a hybrid algorithm for estimating parameters in hybrid dynamical systems of the form
\begin{equation} \label{eqn:PEplant}
\begin{aligned} 
    \dot x &= f_c(x,u(t,j)) + \phi_c(t,j) \theta 
    & (x,u(t,j)) &\in C_P \\
    x^+ &= g_d(x,u(t,j)) + \phi_d(t,j) \theta
    & (x,u(t,j)) &\in D_P,
\end{aligned}
\end{equation}
where $x \in \reals^n$ is the known state vector and $(x,u) \mapsto f_c(x,u) \in \reals^n$ and $(x,u) \mapsto g_d(x,u) \in \reals^n$ are known continuous functions. The regressors $(t,j) \mapsto \phi_c(t,j) \in \reals^{n \times p}$ and $(t,j) \mapsto \phi_d(t,j) \in \reals^{n \times p}$ and the input $(t,j) \mapsto u(t,j) \in \reals^m$ are known, and are defined on hybrid time domains as described in Section~\ref{sec:COMHybridSystems}, but are not necessarily hybrid arcs.\footnote{In other words, $\phi_c$, $\phi_d$, and $u$ do not need to be locally absolutely continuous during flows -- see \cite{220} for details.} The flow set is $C_P \subset \reals^n \times \reals^m$, the jump set is $D_P \subset \reals^n \times \reals^m$, and $n, p, m \in \nats$.
Note that $\phi_c$ plays no role in the dynamics of \eqref{eqn:PEplant} at jumps and $\phi_d$ plays no role in the dynamics of \eqref{eqn:PEplant} during flows. 

Our goal is to estimate the parameter vector $\theta$ in \eqref{eqn:PEplant}.
Since $\phi_c$ and $\phi_d$ may exhibit both flows and jumps, it is important to update the parameter estimate $\hat \theta$ continuously whenever $\phi_c$ flows, and discretely each time $\phi_d$ jumps, which is possible when jumps are detected instantaneously. Hence, we propose to estimate $\theta$ using a hybrid algorithm, denoted $\HS_g$, of the form
\begin{equation} \label{eqn:PEHSg}
\HS_g :
\left\{
\begin{aligned}
    \dot \xi &= F_g(\xi)
    \qquad& \xi &\in C_g  \\
    \xi^+ &= G_g(\xi)
    \qquad& \xi &\in D_g,
\end{aligned}
\right.
\end{equation}
with data designed to solve the following problem. \newline

\noindent {\bf Problem Statement:} Design the data $(C_g,F_g,D_g,G_g)$ of $\HS_g$ in \eqref{eqn:PEHSg} and determine conditions on $\phi_c$ and $\phi_d$ that ensure the parameter estimate $\hat \theta$ converges to the unknown parameter vector $\theta$ in \eqref{eqn:PEplant}. \newline
%

Next, we present our solution to this problem.
\subsection{Problem Solution}
Given $\phi_c, \phi_d : E \to \reals^{n \times p}$ and $u : E \to \reals^m$, where $E := \dom \phi_c = \dom \phi_d = \dom u$ is a hybrid time domain,
we define the state $\xi$ of $\HS_g$ as $\xi := (x, \hat \theta, \psi, \eta, \tau, k) \in \mathcal{X}_g := \mathbb{R}^n \times \reals^p \times \mathbb{R}^{n \times p} \times \mathbb{R}^n \times E$,
where $x$ is the state of the plant in \eqref{eqn:PEplant}, $\hat \theta$ is the estimate of $\theta$, and $\psi, \eta$ are auxiliary state variables. 
The state components $\tau$ and $k$ have dynamics such that they evolve as $t$ and $j$, respectively, from the hybrid time domain $E$.
Including $\tau$ and $k$ in $\xi$ allows $\phi_c$, $\phi_d$, and $u$ to be part of the definitions of $F_g$ and $G_g$, rather than modeled as inputs to $\HS_g$.
Thus, we can express $\HS_g$ as an autonomous hybrid system, which allows us to leverage recent results on stability and robustness properties \cite{65,220}.

During flows, we update $\hat \theta$ with dynamics inspired by the continuous-time algorithm in \eqref{eqn:PEtemp1} and \eqref{eqn:PEhatthetaC},
\begin{align*} 
    \dot{\hat \theta} &= \gamma_c \psi^\top (y - \psi \hat \theta),
\end{align*}
where
%
    $y := x + \eta$,
%
with $\psi$ and $\eta$ generated by
$\dot \psi = -\lambda_c \psi + \phi_c(\tau,k)$ and
$\dot \eta = -\lambda_c (x + \eta) - f_c(x,u(\tau,k))$,
%
%
and $\gamma_c, \lambda_c > 0$ are design parameters.
Hence, for all $\xi \in C_g$, the flow map for $\HS_g$ in \eqref{eqn:PEHSg} is
\begin{equation*} 
\begin{aligned}
    F_g(\xi) &:= 
    \begin{bmatrix}   
        f_c(x,u(\tau,k)) + \phi_c(\tau,k) \theta \\
        \gamma_c \psi^\top (y - \psi \hat \theta) \\
        - \lambda_c \psi + \phi_c(\tau,k) \\
        -\lambda_c (x + \eta) - f_c(x,u(\tau,k)) \\
        1 \\
        0
    \end{bmatrix}.
\end{aligned}
\end{equation*}
At jumps, we update $\hat \theta$ using a reset map inspired by the discrete-time algorithm in \eqref{eqn:PEhatthetaD} and \eqref{eqn:PEpsietaplus},
\begin{equation*} 
    \hat \theta^+ = \hat \theta + \frac{\psi^{+\top}}{\gamma_d + |\psi^+|^2} ( y^+ - \psi^+ \hat \theta ) ,
\end{equation*}
with $\psi^+ := (1-\lambda_d) \psi + \phi_d(\tau,k)$, $\eta^+ := (1-\lambda_d) (x + \eta) - g_d(x,u(\tau,k))$, and $y^+ := x^+ + \eta^+$, where $x^+$ gives the plant state $x$ after a jump per \eqref{eqn:PEplant}, and $\gamma_d >0$, $\lambda_d \in (0, 2)$ are design parameters.
Hence, for all $\xi \in D_g$, the jump map for $\HS_g$ in \eqref{eqn:PEHSg} is
\begin{equation} \label{eqn:PEGg}
    \hspace{-1mm}
    \begin{aligned}
    G_g(\xi) &:=
    \begin{bmatrix}
        g_d(x,u(\tau,k)) + \phi_d(\tau,k) \theta \\
        \hat \theta + \frac{\psi^{+\top}}{\gamma_d + |\psi^{+}|^2} ( y^+ - \psi^{+} \hat \theta ) \\
        (1-\lambda_d) \psi + \phi_d(\tau,k) \\
        (1-\lambda_d) (x + \eta) - g_d(x,u(\tau,k)) \\
        \tau \\
        k + 1
    \end{bmatrix}.
\end{aligned}
\hspace{-1mm}
\end{equation}
The flow and jump sets of $\HS_g$ are defined so that the algorithm flows when $\phi_c$ flows, and jumps when $\phi_d$ jumps. Since $\dom \phi_c = \dom \phi_d = E$,
\begin{equation} \label{eqn:PECgDg}
\hspace{-2.5mm}
\mbox{$\squeezespaces{0.6}
\begin{aligned} 
    C_g &:= \cl \left( {\cal X}_g \setminus D_g \right), &
    D_g &:= \{\xi \in {\cal X}_g : (\tau,k+1) \in E \}.
\end{aligned}
$}
\hspace{-3mm}
\end{equation}
%
%
\begin{remark}
    We assume for simplicity that the plant state $x$ has the same hybrid time domain as $\phi_c$, $\phi_d$, and $u$. As a result, the flow set $C_P$ and jump set $D_P$ of the plant are not part of the construction of $\HS_g$. Our algorithm can be extended to the case where $x$, $\phi_c$, $\phi_d$, and $u$ have different hybrid time domains by considering the flow and jump sets in \eqref{eqn:PEplant}. In this case, we can reparameterize the domains of $\phi_c$, $\phi_d$, and $u$ to express $x$, $\phi_c$, $\phi_d$, and $u$ on a common hybrid time domain. See, e.g., \cite{209}.
\end{remark}
\begin{remark} \label{rmk:PEdoublejump}
    For simplicity, the hybrid algorithm ${\cal H}_g$ in \eqref{eqn:PEHSg} is expressed such that jumps in the parameter estimate coincide with jumps in $x$. This results in ${\cal H}_g$ being noncausal since measurements of $x^+$ are not available until after a jump. 
    %
    %
    We can remove the simplification at the price of letting the algorithm jump twice for each jump in $x$, as follows.
    %
    %
    Immediately before a jump in $x$, the algorithm jumps once to reset the values of $\psi$, $\eta$, and $k$ per the jump map \eqref{eqn:PEGg}. 
    Immediately after a jump in $x$, the algorithm jumps a second time to update the parameter estimate using the current value of $x$.
    A logic variable ensures that, after the second jump, the algorithm flows or jumps in accordance with the hybrid time domain $E$.
    %
    %
    %
    Since $\theta$ in \eqref{eqn:PEplant} is constant,
    the stability properties induced by $\HS_g$ in \eqref{eqn:PEHSg} are equivalent to the stability properties
    induced by the causal modification,
    after we reparameterize the domain of solutions to $\HS_g$ to match the domain of solutions to the causal system.
    Hence, for simplicity, we focus our analysis on \eqref{eqn:PEHSg}.
\end{remark}

\section{Stability Analysis} 
\label{sec:PEStabilityAnalysis}
We now establish our main stability result stating conditions that ensure the hybrid system $\HS_g$ in \eqref{eqn:PEHSg} induces semiglobal pre-exponential stability\footnote{
Since each solution $\HS_g$ inherits the hybrid time domain of $\phi_c$, $\phi_d$, and $u$, the use of ``pre-exponential,'' as opposed to ``exponential,'' stability means that $\phi_c$, $\phi_d$, and $u$ do not need to be complete.
}
of the set
\begin{align} \label{eqn:PEcalAg}
    {\cal A}_g := \defset{\xi \in {\cal X}_g}{\hat \theta = \theta, \; \varepsilon = 0},
\end{align}
where
\vspace{-5mm}
\begin{align} \label{eqn:PEvarpeilonHSg}
    \varepsilon := x + \eta - \psi \theta.
\end{align}
Semiglobal pre-exponential stability of ${\cal A}_g$ implies that, given any compact set of initial conditions, for each solution $\xi$ to $\HS_g$ from such compact set, the distance from $\xi$ to the set ${\cal A}_g$ is bounded above by an exponentially decreasing function of the initial condition -- see Definition \ref{def:COMSemiGlobExpStab}. As a consequence, for each complete solution $\xi$ to $\HS_g$, the parameter estimate $\hat \theta$ converges exponentially to $\theta$, and $\varepsilon$ converges exponentially to zero.
\begin{theorem} \label{thm:PEGlobalStabilityHSg}
Given the hybrid system $\HS_g$ in \eqref{eqn:PEHSg}, $\gamma_c, \lambda_c, \gamma_d > 0$, and $\lambda_d \in (0, 2)$, suppose that $\phi_c, \phi_d : E \to \reals^{n \times p}$ satisfy the following:
\begin{enumerate}[label=\arabic*.]
    \item There exists $\phi_M \in \realsg$ such that $|\phi_c(t,j)|_{\mathrm F} \leq \phi_M$ for all $(t,j) \in E$ and $|\phi_d(t,j)|_{\mathrm F} \leq \phi_M$ for all $(t,j) \in \Upsilon(E)$.
    \item \sloppy There exist $\Delta, \mu \in \realsg$ such that, for all $(t',j'), (t^*,j^*) \in E$ satisfying\footnote{
    The hybrid time instants $(t', j')$ and $(t^*,j^*)$ are the beginning and the end, respectively, of a hybrid time interval with length satisfying \eqref{eqn:PEDeltaCondition}, over which \eqref{eqn:PEHybridPEg} holds.
    }
    \begin{equation}  \label{eqn:PEDeltaCondition}
    \vspace{-0.5mm}
    \begin{aligned}
        \Delta \leq (t^* - t') + (j^* - j') < \Delta + 1,
    \end{aligned}
    \vspace{-0.5mm}
    \end{equation}
    the following hybrid PE condition holds:
    %
    %
    %
    \begin{equation} \label{eqn:PEHybridPEg}
    \begin{aligned} 
        &\sum_{j = j'}^{j^*} \int_{\max\{t',t_j\}}^{\min\{t^*,t_{j+1}\}} \psi(s,j)^\top \psi(s,j) ds \\
        &+ \sum_{j = j'}^{j^*-1} \psi(t_{j+1},j+1)^\top \psi(t_{j+1},j+1)
        \geq \mu I
    \end{aligned}
    \end{equation}
    where $\{t_j\}_{j=0}^{J}$ is the sequence defining $E$ as in Section~\ref{sec:COMHybridSystems}, $t_{J+1} := T$, with $J := \sup_j E$ and $T := \sup_t E$, and $(t,j) \mapsto \psi(t,j)$ is generated by \eqref{eqn:PEHSg}.
\end{enumerate}
Then, for each $\psi_0 \geq 0$, $q_M \geq q_m > 0$, and each $\zeta \in (0, 1)$, there exist $\kappa_g, \lambda_g > 0$ such that each solution $\xi$ to $\HS_g$ from $\xi(0,0) \in {\cal X}_0 := \{\xi \in {\cal X}_g : |\psi|_{\mathrm F} \leq \psi_0 \}$ satisfies
\begin{align} \label{eqn:PEExpStabilityAg2}
    |\xi(t,j)|_{{\cal A}_g} \leq \kappa_g \mathrm{e}^{-\lambda_g(t + j)} |\xi(0,0)|_{{\cal A}_g}
\end{align}
for all $(t,j) \in \dom \xi$. In particular, suitable choices of $\kappa_g$ and $\lambda_g$ are given in Appendix~\ref{apx:PEDefKappagLambdag}.
%
%
%
%
\end{theorem}
%
%
Theorem~\ref{thm:PEGlobalStabilityHSg} states that, if $|\phi_c|_{\mathrm F}$ and $|\phi_d|_{\mathrm F}$ are uniformly bounded above and the hybrid PE condition \eqref{eqn:PEHybridPEg} is satisfied, then the set ${\cal A}_g$ in \eqref{eqn:PEcalAg} is semiglobally pre-exponentially stable for $\HS_g$. 
The hybrid PE condition \eqref{eqn:PEHybridPEg} reduces to the continuous-time PE condition \eqref{eqn:COMPEContinuous} if $\psi$ is continuous,
and reduces to the discrete-time PE condition \eqref{eqn:COMPEDiscrete} if $\psi$ is discrete.
Hence, in such cases, we recover the results established in \cite{Narendra_adaptive_1989,Tao_adaptive_2003}.\footnote{
In fact, if $\psi$ is continuous and $\phi_c$, $\frac{d}{d t} \phi_c \in {\cal L}_\infty$, then it follows from \cite[Lemma~2.6.7]{Sastry1989Adaptive} that the $\psi$ component of each solution $\xi$ to $\HS_g$ from ${\cal X}_0$ is PE as in \ref{C1} if $\phi_c$ is PE.
Given such $\phi_c$, the excitation parameters for $\psi$ -- $\mu$ and $T$ in \ref{C1} -- depend on the initial condition of $\psi$. However, since $\xi(0,0) \in {\cal X}_0$, the initial condition of $\psi$ lies in a compact set, and therefore we can find these parameters independent of the initial condition.
If $\psi$ is discrete, then a similar persistence of excitation property holds for $\psi$ if $\phi_d$ is PE as in \ref{C2}.
}

\subsection{Proof of Theorem~\ref{thm:PEGlobalStabilityHSg}} \label{sec:PEMainProof}
The proof of Theorem~\ref{thm:PEGlobalStabilityHSg} proceeds as follows. 
In Section~\ref{sec:PEHybGradClass}, we generalize the error dynamics of ${\HS}_g$ to a class of hybrid systems, denoted by $\HS$. Section~\ref{sec:PEStabilityHS} establishes conditions on the data of $\HS$ that ensure global pre-exponential stability of a closed set for $\HS$. Then, in Section~\ref{sec:PEStabilityHSg}, we show that the conditions of Theorem~\ref{thm:PEGlobalStabilityHSg} are sufficient to ensure that ${\HS}_g$ satisfies the conditions imposed on $\HS$ in Section~\ref{sec:PEStabilityHS}. Under such conditions, ${\HS}_g$ inherits the stability properties of ${\cal H}$.

\subsubsection{A General Class of Hybrid Gradient Algorithms} \label{sec:PEHybGradClass}
Convergence to $\theta$ for the solution component $\hat \theta$ of $\HS_g$ in \eqref{eqn:PEHSg} is achieved when the parameter estimation error $\tilde \theta = \theta - \hat \theta$ of $\HS_g$ converges to zero.
%
%
%
We denote the hybrid system resulting from expressing $\HS_g$ in error coordinates as $\widetilde{\HS}_g$,
with state $\xi = (x, \tilde \theta, \psi, \eta, \tau, k) \in {\cal X}_g$ and dynamics
\begin{equation} \label{eqn:PEHSgerr}
\widetilde{\HS}_g :
\left\{
\begin{aligned}
    \dot \xi &= \widetilde{F}_g(\xi)
    \qquad& \xi &\in \widetilde{C}_g  \\
    \xi^+ &= \widetilde{G}_g(\xi)
    \qquad& \xi &\in \widetilde{D}_g,
\end{aligned}
\right.
\end{equation}
where $\widetilde{C}_g := C_g$ and $\widetilde{D}_g := D_g$, with $C_g$, $D_g$ in \eqref{eqn:PECgDg}, and
\begin{align*}
\widetilde{F}_g(\xi) &:=
\begin{bmatrix}   
    f_c(x,u(\tau,k)) + \phi_c(\tau,k) \theta \\
    - \gamma_c \psi^\top\psi \tilde \theta - \gamma_c \psi^\top \varepsilon \\
    - \lambda_c \psi + \phi_c(\tau,k) \\
    -\lambda_c (x + \eta) - f_c(x,u(\tau,k)) \\
    1 \\
    0
\end{bmatrix}
\quad \forall \xi \in \widetilde{C}_g \\
\widetilde{G}_g(\xi) &:=
\begin{bmatrix}
    g_d(x,u(\tau,k)) + \phi_d(\tau,k) \theta \\
    \tilde \theta - \frac{\psi^{+\top} \psi^{+}}{\gamma_d + |\psi^{+}|^2} \tilde \theta - \frac{\psi^{+\top}}{\gamma_d + |\psi^{+}|^2} \varepsilon^+ \\
    (1-\lambda_d) \psi + \phi_d(\tau,k) \\
    (1-\lambda_d) (x + \eta) - g_d(x,u(\tau,k)) \\
    \tau \\
    k + 1
\end{bmatrix}
\quad \forall \xi \in \widetilde{D}_g,
\end{align*}
%
%
with $\varepsilon$ as in \eqref{eqn:PEvarpeilonHSg} and $\varepsilon^+ := x^+ + \eta^+ - \psi^+ \theta$, where $x^+$ gives the plant state $x$ after a jump per \eqref{eqn:PEplant}. %

To analyze the stability properties induced by $\widetilde{\HS}_g$, we use that $\widetilde{\HS}_g$ in \eqref{eqn:PEHSgerr} belongs to a class of hybrid systems,
denoted by $\HS$, with state $\xi = (\vartheta,\tau,k) \in {\cal X} := \reals^p \times E$ and dynamics
\begin{equation} \label{eqn:COMHS}
    \hspace{-2mm}
    \begin{aligned}
        \dot \xi &=
        \begin{bmatrix}
        -A(\tau,k) \vartheta + d_c(\tau,k) \\
        1 \\
        0
        \end{bmatrix} =: F(\xi)
        & \xi &\in C  \\
        \xi^+ &= 
        \begin{bmatrix}
        \vartheta - B(\tau,k) \vartheta + d_d(\tau,k) \\
        \tau \\
        k + 1
        \end{bmatrix} =: G(\xi)
        & \xi &\in D
    \end{aligned}
    \hspace{-7mm}
\end{equation}
where $A, B: E \to \reals^{p \times p}$ and $d_c, d_d: E \to \reals^{p}$ are given and $E := \dom A = \dom B = \dom d_c = \dom d_d$ is a hybrid time domain, $C := \cl \left( {\cal X} \setminus D \right)$, and $D := \{\xi \in {\cal X} : (\tau,k+1) \in E \}$.
\begin{remark} \label{rem:equivalence}
    %
    The hybrid system $\HS$ in \eqref{eqn:COMHS} reduces to $\widetilde{\HS}_g$ in \eqref{eqn:PEHSgerr} when $\vartheta = \tilde \theta$,
    %
    \begin{subequations} \label{eqn:PEABd}
    \begin{align}  
        A(\tau,k) &= \gamma_c \psi(\tau,k)^\top \psi(\tau,k) \label{eqn:PEA} \\
        d_c(\tau,k) &= -\gamma_c \psi(\tau,k)^\top \varepsilon(\tau,k) \label{eqn:PEd_c}
        \intertext{for all $(\tau,k) \in E$, and\footnotemark}\noalign{\footnotetext{
        Note that $B$ is evaluated only at jump times in \eqref{eqn:COMHS}, and $B(\tau,k)$ in \eqref{eqn:PEB} is well defined for all $(\tau,k) \in \Upsilon(E)$. Furthermore, the expression for $d_d$ in \eqref{eqn:PEd_d} includes the value of the disturbance $\varepsilon$ after a jump, which results in a noncausal algorithm -- see Remark \ref{rmk:PEdoublejump}.
        }}
        B(\tau,k) &= \frac{\psi(\tau,k+1)^{\top} \psi(\tau,k+1)}{\gamma_d + | \psi(\tau,k+1)|^2} \label{eqn:PEB} \\
        d_d(\tau,k) &= -\frac{\psi(\tau,k+1)^{\top}}{\gamma_d + |\psi(\tau,k+1)|^2} \varepsilon(\tau,k+1) \label{eqn:PEd_d}
    \end{align}
    \end{subequations}
    for all $(\tau,k) \in \Upsilon(E)$, with $\Upsilon$ as in \eqref{eqn:COMUpsilon}, 
    where $\varepsilon = x + \eta - \psi \theta$ is a hybrid disturbance and $x, \eta, \psi$ satisfy the dynamics in \eqref{eqn:PEHSgerr}.
\end{remark}
We impose on $A$ and $B$ the following structural properties, which are similar to those imposed in the design of continuous-time and discrete-time gradient algorithms.
\begin{assumption} \label{asm:COMABStructure}
Given $A, B : E \to \reals^{p \times p}$, where $E := \dom A = \dom B$ is a hybrid time domain,
\begin{enumerate}[label=\arabic*.]
    \item~$A(t,j) = A(t,j)^\top \geq 0$ for all $(t,j) \in E$;
    %
    \item~$B(t,j) = B(t,j)^\top \geq 0$ for all $(t,j) \in \Upsilon(E)$;
    %
    \item~there exists $a_M > 0$ such that $\text{\rm ess sup } \{|A(t,j)| : (t,j) \in E \} \leq a_M$;
    %
    %
    \item~$|B(t,j)| < 1$ for all $(t,j) \in \Upsilon(E)$.
    %
\end{enumerate}
\end{assumption}
%
%
We impose the following hybrid PE condition \cite{242}.%
\begin{assumption} \label{asm:COMHybridPE}
Given $A, B : E \to \reals^{p \times p}$, where $E := \dom A = \dom B$ is a hybrid time domain,
there exist $\Delta, \mu_0 \in \realsg$ such that, for each $(t',j'), (t^*,j^*) \in E$ satisfying 
%
    $\Delta \leq (t^* - t') + (j^* - j') < \Delta + 1$,
%
\iftoggle{arxiv}{the following holds:}{}
\begin{equation} \label{eqn:COMHybridPE}
\hspace{-3mm}
\mbox{$\squeezespaces{0.57}
\begin{aligned} 
    \sum_{j = j'}^{j^*} \int_{\max\{t',t_j\}}^{\min\{t^*,t_{j+1}\}}  A(s,j) ds
    + \frac{1}{2} \sum_{j = j'}^{j^*-1} B(t_{j+1},j)
    \geq \mu_0 I
\end{aligned}
$}
\hspace{-2mm}
\end{equation}
%
%
where $t_{J+1} := T$, with $J := \sup_j E$ and $T := \sup_t E$.
\end{assumption}

\subsubsection{Stability Analysis for $\HS$} \label{sec:PEStabilityHS}

In this section, we establish sufficient conditions on $A, B, d_c,$ and $d_d$ that ensure the hybrid system $\HS$ induces global pre-exponential stability of the set
\begin{align} \label{eqn:PEcalA}
    {\cal A} := \defset{\xi \in {\cal X}}{\vartheta = 0}.
\end{align}
%
%
We first establish the following ISS result for $\HS$.
%
%
%
%
%
\begin{theorem} \label{thm:COMISS}
Given the hybrid system $\HS$ in \eqref{eqn:COMHS},  
let Assumptions~\ref{asm:COMABStructure} and \ref{asm:COMHybridPE} hold.
Then, for each $q_M \geq q_m > 0$ and each $\zeta \in (0, 1)$, each solution $\xi$ to $\HS$ satisfies
\begin{align} \label{eqn:COMISS}
    |\xi(t,j)|_{\cal A} \leq \beta(|\xi(0,0)|_{\cal A},t+j) + \rho \| d \|_{(t,j)}
\end{align}
for all $(t,j) \in \dom \xi$, where
\begin{align} \label{eqn:COMd}
    d(t,j) &:= 
    \left\{
    \begin{aligned}
        &d_c(t,j) & &\text{\rm if} & (t,j) &\in E \setminus \Upsilon(E) \\
        &d_d(t,j) & &\text{\rm if} & (t,j) &\in \Upsilon(E)
    \end{aligned}
    \right.
\end{align}
and
\begin{align*} 
    \beta(s,r) &:= \sqrt{\frac{p_M}{p_m}} \mathrm{e}^{-\omega r} s, \quad
    \rho := \sqrt{\frac{2 p_M^3}{q_m p_m \zeta} \left( \frac{2 p_M}{q_m} + 1 \right)} \notag \\
    \omega &:= \mbox{$\squeezespaces{0.6}\displaystyle \frac{1}{2} \min \left\{ \frac{q_m}{2 p_M} (1 - \zeta), \,
    -\ln \left( 1 - \frac{q_m}{2 p_M} (1 - \zeta) \right) \right\},$} \notag \\
    p_m &:= q_m, \qquad\quad \ \
    p_M := q_m + \frac{q_M \kappa_0^2}{2 \lambda_0} + \frac{q_M \kappa_0^2 \mathrm{e}^{2 \lambda_0}}{\mathrm{e}^{2 \lambda_0} - 1}, \notag \\
    \kappa_0 &:= \sqrt{\frac{1}{1 - \sigma}}, \qquad \
    \lambda_0 := - \frac{\ln (1 - \sigma)}{2 (\Delta + 1)}, \notag \\
    \sigma &:= \mbox{$\squeezespaces{0.6}\displaystyle \frac{2 \mu_0}{\big( 1 + \sqrt{(a_M + 2) (\Delta + 2)^3 (a_M (\Delta + 2) + 1/2)} \big)^{2}},$} \notag
\end{align*}
with $\Upsilon$ defined in \eqref{eqn:COMUpsilon} and 
%
%
%
%
%
%
$a_M, \mu_0, \Delta$ from Assumptions~\ref{asm:COMABStructure} and \ref{asm:COMHybridPE}.
%
%
\end{theorem}
%

%
%

Motivated by the fact that, for each complete solution to ${\HS}_g$ in \eqref{eqn:PEHSg}, the signal $(t,j) \mapsto \varepsilon(t,j)$ in \eqref{eqn:PEABd} converges exponentially to zero, we use Theorem~\ref{thm:COMISS} to establish the stability properties induced by $\HS$ when $d_c$ and $d_d$ converge exponentially to zero.
\begin{theorem} \label{thm:COMISS2}
Given the hybrid system $\HS$ in \eqref{eqn:COMHS}, suppose that Assumptions~\ref{asm:COMABStructure} and \ref{asm:COMHybridPE} hold, and that there exist $a, b > 0$ such that $d$ in \eqref{eqn:COMd} satisfies
\begin{align} \label{eqn:PEdbound}
    |d(t,j)| \leq a \mathrm{e}^{-b (t + j)} |d(0,0)|
\end{align}
for all $(t,j) \in E$. 
Then, for each $q_M \geq q_m > 0$ and each $\zeta \in (0, 1)$, each solution $\xi$ to $\HS$ satisfies
\begin{align} \label{eqn:PExibound}
    |\xi(t,j)|_{\cal A} \leq \kappa \mathrm{e}^{-\lambda (t + j)} \big( |\xi(0,0)|_{\cal A} + |d(0,0)| \big)
\end{align}
for all $(t,j) \in \dom \xi$, where
\begin{align} \label{eqn:PEkappalambda}
    \kappa &:= 
    2 \max \left\{ 
    \dfrac{p_M}{p_m}, \,
    a \rho \sqrt{\dfrac{p_M}{p_m}}
    \right\}, &
    \lambda &:= \frac{1}{2} \min \left\{ \omega, b \right\}
\end{align}
with $p_m, p_M, \rho,$ and $\omega$ from Theorem~\ref{thm:COMISS}.
\end{theorem}
%

We use Theorem~\ref{thm:COMISS2} in the next section to prove the stability properties induced by our algorithm.

\subsubsection{Stability Analysis for $\HS_g$} \label{sec:PEStabilityHSg}

To prove Theorem~\ref{thm:PEGlobalStabilityHSg}, we require the following results for the error dynamics of our algorithm.
\begin{lemma} \label{prop:PEepsStab}
\sloppy Given the hybrid system $\widetilde{\HS}_g$ in \eqref{eqn:PEHSgerr}, for each $\lambda_c > 0$, $\lambda_d \in (0, 2)$, and each solution $\xi = (x, \tilde \theta, \psi, \eta, \tau, k)$ to $\widetilde{\HS}_g$, $(t,j) \mapsto \varepsilon(t,j) := x(t,j) + \eta(t,j) - \psi(t,j) \theta$ in \eqref{eqn:PEvarpeilonHSg} satisfies
\begin{align} \label{eqn:PEepsbound}
    |\varepsilon(t,j)| \leq \mathrm{e}^{-b(t + j)} |\varepsilon(0,0)|
    \quad \forall (t,j) \in \dom \xi,
\end{align}
where 
%
    $b := \frac{1}{2} \min \left\{ 
    2 \lambda_c, \ - \ln \left( 1 - \lambda_d ( 2 - \lambda_d ) \right)
    \right\}$.
%
\end{lemma}
%
%
%
\begin{lemma} \label{prop:PEpsiStab}
Given the hybrid system $\widetilde{\HS}_g$ in \eqref{eqn:PEHSgerr}, suppose that $\phi_c : E \to \reals^{n \times p}$ satisfies item~1 of Theorem~\ref{thm:PEGlobalStabilityHSg} and let $\phi_M > 0$ come from that item. Then, for each $\psi_0 \geq 0$, $\lambda_c > 0$, $\lambda_d \in (0,2)$, the $\psi$ component of each solution $\xi$ to $\widetilde{\HS}_g$ from $\xi(0,0) \in {\cal X}_0 := \{\xi \in {\cal X}_g : |\psi|_{\mathrm F} \leq \psi_0 \}$ satisfies
\begin{align} \label{eqn:PEpsibound}
    |\psi(t,j)| &\leq \psi_M
    \quad \forall (t,j) \in \dom \xi,
\end{align}
where
%
    $\psi_M := \psi_0 + 
    \max \left\{
        \frac{1}{\lambda_c}, 
        \frac{\sqrt{2 \lambda_d (2 - \lambda_d) + 16}}{\lambda_d (2 - \lambda_d)}
    \right\}
    \phi_M$.
%
\end{lemma}
%
%
%
We now have all the ingredients to prove Theorem~\ref{thm:PEGlobalStabilityHSg}. 
%

\begin{proof}[\unskip\nopunct]{\bf Proof of Theorem~\ref{thm:PEGlobalStabilityHSg}:}
To prove Theorem~\ref{thm:PEGlobalStabilityHSg}, we show that the error dynamics of $\HS_g$ -- that is, $\widetilde{\HS}_g$ in \eqref{eqn:PEHSgerr} -- satisfy the conditions of Theorem~\ref{thm:COMISS2} with $A, B, d_c, d_d$ in \eqref{eqn:PEABd}.
Beginning with Assumption~\ref{asm:COMABStructure}, since $A$, $B$ in \eqref{eqn:PEABd} are symmetric and $\gamma_c, \gamma_d > 0$, it follows that they are positive semidefinite.
Hence, items~1 and~2 of Assumption~\ref{asm:COMABStructure} holds.
Next, we show that items~3 and~4 of Assumption~\ref{asm:COMABStructure} hold. Since, by item~1 in Theorem~\ref{thm:PEGlobalStabilityHSg}, the conditions of Lemma~\ref{prop:PEpsiStab} are satisfied, it follows that, for each solution $\xi$ to $\widetilde{\HS}_g$ from ${\cal X}_0$, the $\psi$ component of $\xi$ satisfies \eqref{eqn:PEpsibound}. Thus,
%
    $|A(t,j)| \leq \gamma_c |\psi(t,j)|^2 \leq \gamma_c \psi_M^2$
%
for all $(t,j) \in E$, with $\psi_M$ from Lemma~\ref{prop:PEpsiStab}, 
and
%
    $|B(t,j)| = \frac{|\psi(t,j+1)|^2}{\gamma_d + |\psi(t,j)|^2} < 1$
%
for all $(t,j) \in \Upsilon(E)$. 
Hence, items~3 and~4 of Assumption~\ref{asm:COMABStructure} hold with $a_M := \gamma_c \psi_M^2$.

%
%
Next, using Lemma~\ref{prop:PEpsiStab} and item 2 in Theorem~\ref{thm:PEGlobalStabilityHSg}, we show that Assumption~\ref{asm:COMHybridPE} holds with $A$, $B$ in \eqref{eqn:PEABd}. Substituting $A$, $B$ into \eqref{eqn:COMHybridPE}, we have that, for all $(t',j'),(t^*,j^*) \in E$ satisfying \eqref{eqn:PEDeltaCondition},
\begin{align*} 
    &\sum_{j = j'}^{j^*} \int_{\max\{t',t_j\}}^{\min\{t^*,t_{j+1}\}} \gamma_c \psi(s,j)^\top \psi(s,j) ds \\
    &\quad+ \frac{1}{2} \sum_{j = j'}^{j^*-1} \frac{\psi(t_{j+1},j+1)^\top \psi(t_{j+1},j+1)}{\gamma_d + | \psi(t_{j+1},j+1)|^2} \\
    &\geq \min \left\{ \gamma_c, \frac{1}{2(\gamma_d + \psi_M^2)} \right\} \mu I.
\end{align*}
Hence, Assumption~\ref{asm:COMHybridPE} holds with $\Delta$ from item 2 of Theorem~\ref{thm:PEGlobalStabilityHSg} and $\mu_0 := \min \left\{ \gamma_c, \frac{1}{2(\gamma_d + \psi_M^2)} \right\} \mu$.

Finally, we show that \eqref{eqn:PEdbound} is satisfied with $d$ in \eqref{eqn:COMd} and $d_c, d_d$ in \eqref{eqn:PEABd}.
By item~1 of Theorem~\ref{thm:PEGlobalStabilityHSg}, it follows from Lemmas~\ref{prop:PEepsStab} and \ref{prop:PEpsiStab} that, for each solution $\xi$ to $\widetilde{\HS}_g$ from ${\cal X}_0$,
%
$|d_c(t,j)| 
\leq \gamma_c |\psi(t,j)| |\varepsilon(t,j)|
\leq \gamma_c \psi_M \mathrm{e}^{-b(t + j)} |\varepsilon(0,0)|$
%
for all $(t,j) \in \dom \xi$, with $\psi_M$ from Lemma~\ref{prop:PEpsiStab} and $b$ from Lemma~\ref{prop:PEepsStab}.
Furthermore, using that
$\frac{|\psi(t,j+1)|}{\gamma_d + |\psi(t,j+1)|^2} \leq \frac{1}{2 \sqrt{\gamma_d}}$ for all $(t,j) \in \Upsilon(\dom \xi)$, we have that 
%
    $|d_d(t,j)| 
    \leq \frac{1}{2 \sqrt{\gamma_d}} |\varepsilon(t,j+1)| 
    \leq \frac{1}{2 \sqrt{\gamma_d}} \mathrm{e}^{-b(t+j+1)} |\varepsilon(0,0)| 
    \leq \frac{1}{2 \sqrt{\gamma_d}} \mathrm{e}^{-b(t+j)} |\varepsilon(0,0)|$
%
for all $(t,j) \in \Upsilon(\dom \xi)$.
Thus, we conclude that \eqref{eqn:PEdbound} holds with $a := \max \left\{ \gamma_c \psi_M, \frac{1}{2 \sqrt{\gamma_d}} \right\}$, $b$ from Lemma~\ref{prop:PEepsStab}, 
and $|d(0,0)| = |\varepsilon(0,0)|$.
Hence, the conditions of Theorem~\ref{thm:COMISS2} hold and, from the equivalence between the data of $\widetilde{\HS}_g$ in \eqref{eqn:PEHSgerr} and $\HS$ in \eqref{eqn:COMHS} with $A, B, d_c, d_d$ in \eqref{eqn:PEABd},\footnote{
In other words, by substituting $A, B, d_c, d_d$ in \eqref{eqn:PEABd} into \eqref{eqn:COMHS} and treating $\psi$ as a given hybrid signal and $\varepsilon$ as hybrid disturbance satisfying \eqref{eqn:PEHSgerr}, we obtain a hybrid system with dynamics that are equivalent to $\widetilde{\HS}_g$ in \eqref{eqn:PEHSgerr}. 
} we have from Theorem~\ref{thm:COMISS2} that the $\tilde \theta$ component of each solution $\xi$ to $\widetilde{\HS}_g$ from ${\cal X}_0$ satisfies
\begin{align} \label{eqn:PEtildethetabound}
    |\tilde \theta(t,j)| \leq \kappa \mathrm{e}^{-\lambda (t + j)} \big( |\tilde \theta(0,0)| + |\varepsilon(0,0)| \big)
\end{align}
for all $(t,j) \in \dom \xi$, with $\kappa, \lambda$ in \eqref{eqn:PEkappalambda}.

To conclude the proof, using the definition of ${\cal A}_g$ in \eqref{eqn:PEcalAg}, we rewrite $|\xi|_{{\cal A}_g}$ for all $(t,j) \in \dom \xi$ as
%
    $|\xi(t,j)|_{{\cal A}_g} = \sqrt{|\tilde \theta(t,j)|^2 + |\varepsilon(t,j)|^2}$.
%
%
Substituting the bounds in \eqref{eqn:PEepsbound} and \eqref{eqn:PEtildethetabound} and using that $\kappa \geq 1$ and, for any $\alpha, \beta \in \reals$, $\alpha \beta \leq \frac{1}{2}(\alpha^2 + \beta^2)$, we conclude that, for all $(t,j) \in \dom \xi$,
\iftoggle{long}{
\begin{align*} 
    |\xi(t,j)|_{{\cal A}_g} 
    &\leq
    \sqrt{3} \kappa \mathrm{e}^{- \min\{\lambda, b\}  (t + j)} \\
    &\quad \times \sqrt{|\tilde \theta(0,0)|^2 + |\varepsilon(0,0)|^2}.
\end{align*}
}{
    $|\xi(t,j)|_{{\cal A}_g} 
    \leq
    \sqrt{3} \kappa \mathrm{e}^{- \min\{\lambda, b\}  (t + j)}
    |\xi(0,0)|_{{\cal A}_g}$.
}
%
%
%
%
\end{proof}
\iftoggle{long}{
\begin{remark}
The stability analysis above invokes that $\widetilde{\HS}_g$ in \eqref{eqn:PEHSgerr} is ISS with respect to $\varepsilon$ in \eqref{eqn:PEvarpeilonHSg}. Establishing a stability result more directly -- that is, without invoking ISS -- is difficult. 
In fact, since the dynamics of $\varepsilon$ are known, an equivalent formulation of $\widetilde{\HS}_g$ has state $\xi := (\tilde \theta, \varepsilon, \psi, \tau, k)$ and dynamics
%
\begin{equation} \label{eqn:PEHSgerrAlt}
\begin{aligned}
    \matt{\dot{\tilde \theta} \\ \dot \varepsilon \\ \dot \psi \\ \dot \tau \\ \dot k} &= 
    \begin{bmatrix}   
        - \gamma_c \psi^\top\psi \tilde \theta - \gamma_c \psi^\top \varepsilon \\
        - \lambda_c \varepsilon \\
        - \lambda_c \psi + \phi_c(\tau,k) \\
        1 \\
        0
    \end{bmatrix}
    & \xi &\in \widetilde{C}_g  \\
    \matt{{\tilde \theta}^+ \\ \varepsilon^+ \\ \psi^+ \\ \tau^+ \\ k^+}  &= 
    \begin{bmatrix}
        \tilde \theta - \frac{\phi_d(\tau,k)^\top \phi_d(\tau,k)}{\gamma_d + |\phi_d(\tau,k)|^2} \tilde \theta \\
        \varepsilon - \lambda_d \varepsilon \\
        \psi - \lambda_d \psi \\
        \tau \\
        k + 1
    \end{bmatrix}
    & \xi &\in \widetilde{D}_g
\end{aligned}
\end{equation}
where 
$\mbox{$\squeezespaces{0.5} \widetilde{C}_g := \cl \big(\widetilde{\cal X}_g \setminus \widetilde{D}_g \big) $}$ and 
$\mbox{$\squeezespaces{0.5}\widetilde{D}_g := \{\xi \in \widetilde{\cal X}_g : (\tau,k+1) \in E \} $}$, with 
$\mbox{$\squeezespaces{0.5} \widetilde{\cal X}_g := \reals^p \times \reals^n \times \mathbb{R}^{n \times p} \times E $}$.
The hybrid system $\HS$ in \eqref{eqn:COMHS} reduces to \eqref{eqn:PEHSgerrAlt} when $\mbox{$\squeezespaces{0.6} x = (\tilde \theta, \varepsilon) $}$ and, for all $\mbox{$\squeezespaces{0.6} (\tau,k) \in E $}$,
\begin{equation} \label{eqn:PEABdAlt}
\begin{aligned}  
    A(\tau,k) &= 
    \matt{\gamma_c \psi(\tau,k)^\top \psi(\tau,k) & \ \ & \gamma_c \psi(\tau,k)^\top \\ 0 & \ \ & \lambda_c I} \\
    B(\tau,k) &= 
    \matt{\frac{\phi_d(\tau,k)^\top \phi_d(\tau,k)}{\gamma_d + | \phi_d(\tau,k)|^2} & & 0 \\ 0 & & \lambda_d I} \\
    d_c(\tau,k) &= d_d(\tau,k) = 0
\end{aligned}
\end{equation}
where $\psi$ is treated as a hybrid input signal generated by \eqref{eqn:PEHSgerrAlt}.
The recent work \cite{242} studies the $\HS$ class of systems when Assumptions~\ref{asm:COMABStructure} and \ref{asm:COMHybridPE} are satisfied and $d_c = d_d = 0$ (see Theorem~\ref{thm:PEMain} in the Appendix). 
However, the results in \cite{242} are not applicable to \eqref{eqn:PEHSgerrAlt} since $A$ in \eqref{eqn:PEABdAlt} is not symmetric and thus violates item~1 of Assumption~\ref{asm:COMABStructure}.
Moreover, it is difficult to extend the analysis in \cite{242} to the case where $A$ is not symmetric.
\end{remark}
}{}

\section{Robustness Analysis} 
\iftoggle{dissertation}{}{\label{sec:PERobustnessAnalysis}}
In this section, we study the robustness properties induced by $\HS_g$ with respect to bounded (hybrid) noise on the state measurements.

Given $\phi_c, \phi_d : E \to \reals^{n \times p}$ and $u : E \to \reals^m$, where $E := \dom \phi_c = \dom \phi_d = \dom u$ is a hybrid time domain, consider additive noise $\nu : E \to \reals^{n}$ in the measurements of the plant state $x$ in \eqref{eqn:PEplant}.\footnote{
For simplicity, we assume that the measurement noise $\nu$ has the same hybrid time domain as $x$, $\phi_c$, $\phi_d$, and $u$.
}
We denote the hybrid system $\HS$ in \eqref{eqn:PEHSg} under the effect of the measurement noise $\nu$ as $\HS_\nu$, with state $\xi = (x, \hat \theta, \psi, \eta, \tau, k) \in {\cal X}_g$ and dynamics
\begin{equation} \label{eqn:PEHSnu}
{\HS}_\nu :
\left\{
\begin{aligned}
    \dot \xi &= {F}_\nu(\xi)
    \qquad& \xi &\in {C}_\nu  \\
    \xi^+ &= {G}_\nu(\xi)
    \qquad& \xi &\in {D}_\nu
\end{aligned}
\right.
\end{equation}
where
\begin{equation*} 
\allowdisplaybreaks
\mbox{$\squeezespaces{0.6}
\begin{aligned}
    {F}_\nu(\xi) &:= 
    \begin{bmatrix}   
        f_c(x,u(\tau,k)) + \phi_c(\tau,k) \theta \\
        \gamma_c \psi^\top (y_{\nu} - \psi \hat \theta) \\
        - \lambda_c \psi + \phi_c(\tau,k) \\
        -\lambda_c (x + \nu(\tau,k) + \eta) - f_c(x + \nu(\tau,k),u(\tau,k)) \\
        1 \\
        0
    \end{bmatrix} \\
    {G}_\nu(\xi) &:=
    \begin{bmatrix}
        g_d(x,u(\tau,k)) + \phi_d(\tau,k) \theta \\
        \hat \theta + \frac{\psi^{+\top}}{\gamma_d + |\psi^{+}|^2} ( y_\nu^+ - \psi^{+} \hat \theta ) \\
        (1-\lambda_d) \psi + \phi_d(\tau,k) \\
        \mbox{$\displaystyle\squeezespaces{0.1}(1-\lambda_d) (x + \nu(\tau,k) + \eta) - g_d(x + \nu(\tau,k),u(\tau,k))$} \\
        \tau \\
        k + 1
    \end{bmatrix}
\end{aligned}
$}
\end{equation*}
%
%
where ${C}_\nu := C_g$, ${D}_\nu := D_g$, with $C_g, D_g$ in \eqref{eqn:PECgDg}, and we define $y_{\nu} := x + \nu(\tau,k) + \eta$ and $y_{\nu}^+ := x^+ + \nu(\tau,k+1) + \eta^+$, where $x^+$ gives the plant state $x$ after a jump per \eqref{eqn:PEplant}.

For analyzing the effect of the noise, we make the following Lipschitz continuity assumption.
\begin{assumption} \label{asm:PEboundedf}
    Given the hybrid plant in \eqref{eqn:PEplant}, there exist $L_c, L_d > 0$ such that, for all $x_1, x_2 \in \reals^n$ and all $u \in \reals^m$,
    \begin{align*}
        |f_c(x_1,u) - f_c(x_2,u)| &\leq L_c |x_1 - x_2|, \\
        |g_d(x_1,u) - g_d(x_2,u)| &\leq L_d |x_1 - x_2|.
    \end{align*}
\end{assumption}
We now establish our main robustness result stating conditions that ensure ${\cal A}_g$ in \eqref{eqn:PEcalAg} is ISS for ${\HS}_\nu$.
%
%
\begin{theorem} \label{thm:PEHgISS}
Given the hybrid system ${\HS}_\nu$ in \eqref{eqn:PEHSnu}, $\gamma_c, \lambda_c, \gamma_d > 0$, and $\lambda_d \in (0, 2)$, suppose that 
\iftoggle{dissertation}{
Assumptions~\ref{asm:PEBoundedPhicPHid},~\ref{asm:PEHybridPEPsi}, and~\ref{asm:PEboundedf} hold. 
}{
Assumption~\ref{asm:PEboundedf} holds and that $\phi_c, \phi_d : E \to \reals^{n \times p}$ satisfy items 1 and 2 of Theorem~\ref{thm:PEGlobalStabilityHSg}. 
}
Then, for each $\psi_0 \geq 0$, $q_M \geq q_m > 0$, and each $\zeta \in (0, 1)$, each solution $\xi$ to ${\HS}_\nu$ from $\xi(0,0) \in {\cal X}_0$ satisfies
\begin{align} \label{eqn:PEISSbound}
    |\xi(t,j)|_{{\cal A}_g} \leq \kappa_\nu \mathrm{e}^{- \lambda_\nu(t+j)} |\xi(0,0)|_{{\cal A}_g} + \rho_\nu d_\nu(t,j)
\end{align}
for all $(t,j) \in \dom \xi$, where
%
%
\begin{align*}
    \kappa_\nu &:= \sqrt{\frac{2 p_M}{p_m}}, \ \ 
    \lambda_\nu := \min\{\omega, \lambda_\varepsilon\}, \ \ 
    \rho_\nu := \sqrt{2} \max\{\rho, \rho_\varepsilon\} \\
    \lambda_\varepsilon &:= \frac{1}{2} \min \left\{
        \lambda_c (1 - \zeta), 
        -\ln \left( 1 - \frac{\lambda_d}{2} (2 - \lambda_d) (1 - \zeta) \right)
        \right\} \\
    \rho_\varepsilon &:= \max \left\{
        \frac{2}{\lambda_c \sqrt{\zeta}}, 
        \frac{\sqrt{2 \lambda_d (2 - \lambda_d) + 16}}{\lambda_d (2 - \lambda_d) \sqrt{\zeta}}
        \right\}
\end{align*}
%
%
\sloppy with $\rho$, $\omega$ from Theorem~\ref{thm:COMISS}, $p_m, p_M$ from Theorem~\ref{thm:PEGlobalStabilityHSg}, $d_\nu(t,j) := \sqrt{\|d\|_{(t,j)}^2 + \|d_\varepsilon\|_{(t,j)}^2}$, with $d$ as in \eqref{eqn:COMd} and
\begin{align} \label{eqn:PEdvarepsilon}
    d_\varepsilon(t,j) &:= 
    \left\{
    \begin{aligned}
        &\alpha_c(t,j) & &\text{\rm if} & (t,j) &\in E \setminus \Upsilon(E) \\
        &\alpha_d(t,j) & &\text{\rm if} & (t,j) &\in \Upsilon(E),
    \end{aligned}
    \right.
\end{align}
where, for all $(t,j) \in E$, 
\begin{subequations} \label{eqn:PEDcDd}
\begin{align} 
    d_c(t,j) &:= - \gamma_c \psi(t,j)^\top ( \varepsilon(t,j) + \nu(t,j) ) \label{eqn:PEDc} \\
    \alpha_c(t,j) &:= -\lambda_c \nu(t,j) + f_c(x(t,j), u(t,j)) \label{eqn:PEalpha1} \\
        &\qquad - f_c(x(t,j) + \nu(t,j), u(t,j)) \notag \\
    \intertext{and, for all $(t,j) \in \Upsilon(E)$,}
    d_d(t,j) &:= - \frac{\psi(t,j+1)^{\top}}{\gamma_d + |\psi(t,j+1)|^2} \big( \varepsilon(t,j+1) \label{eqn:PEDd} \\
    &\qquad + \nu(t,j+1) \big), \notag \\
    \alpha_d(t,j) &:= (1 - \lambda_d) \nu(t,j) + g_d(x(t,j), u(t,j)) \label{eqn:PEalpha2} \\
    &\qquad - g_d(x(t,j) + \nu(t,j), u(t,j)) \notag
\end{align}
\end{subequations}
with $\varepsilon$ as in \eqref{eqn:PEvarpeilonHSg}.
Moreover, for all $(t,j) \in E$,
%
%
%
\begin{align*}
    |d_c(t,j)| &\leq \gamma_c \psi_M \big( \mathrm{e}^{-\lambda_\varepsilon (t + j)} |\varepsilon(0,0)| \\
    &\quad \!\!\!\!\mbox{$\squeezespaces{1}+ (\rho_\varepsilon \max\{\lambda_c+L_c, 1-\lambda_d+L_d\} + 1) |\nu(t,j)| \big) $} \\
    |\alpha_c(t,j)| &\leq (\lambda_c  + L_c) |\nu(t,j)|
    \intertext{and for all $(t,j) \in \Upsilon(E)$,}
    |d_d(t,j)| &\leq \frac{1}{2 \sqrt{\gamma_d}} \big( \mathrm{e}^{-\lambda_\varepsilon (t + j + 1)} |\varepsilon(0,0)| \\ 
    &\quad \!\!\!\!\mbox{$\squeezespaces{0.5}+ (\rho_\varepsilon \max\{\lambda_c+L_c, 1-\lambda_d+L_d\} + 1) |\nu(t,j+1)| \big) $} \\
    |\alpha_d(t,j)| &\leq (1 - \lambda_d + L_d) |\nu(t,j)|
\end{align*}
with $\psi_M$ from \iftoggle{dissertation}{Assumption~\ref{asm:PEBoundedPhicPHid}}{Theorem~\ref{thm:PEGlobalStabilityHSg}}, $L_c$, $L_d$ from Assumption~\ref{asm:PEboundedf}, and $\varepsilon(0,0) = x(0,0) + \eta(0,0) - \psi(0,0) \theta$.
\end{theorem}
To prove Theorem~\ref{thm:PEHgISS}, we require the following result.
\begin{lemma} \label{lem:PEepsISS}
Given the hybrid system ${\HS}_\nu$ in \eqref{eqn:PEHSnu}, suppose that Assumption~\ref{asm:PEboundedf} holds. Then, for each $\lambda_c > 0$, $\lambda_d \in (0, 2)$, $\zeta \in (0,1)$, and each solution $\xi = (x, \hat \theta, \psi, \eta, \tau, k)$ to ${\HS}_\nu$, $(t,j)~\mapsto~\varepsilon(t,j) := x(t,j) + \eta(t,j) - \psi(t,j) \theta$ in \eqref{eqn:PEvarpeilonHSg} satisfies, for all $(t,j) \in \dom \xi$,
\begin{align} \label{eqn:PEepsISS}
    |\varepsilon(t,j)| \leq \mathrm{e}^{-\lambda_\varepsilon (t + j)} |\varepsilon(0,0)|
    + \rho_\varepsilon \|d_\varepsilon\|_{(t,j)}
\end{align}
with $\lambda_\varepsilon, \rho_\varepsilon > 0$ and $(t,j) \mapsto d_{\varepsilon}(t,j)$ from Theorem~\ref{thm:PEHgISS}.
\end{lemma}
\iftoggle{dissertation}{
\begin{proof}
This proof is given in Appendix~\ref{apx:epsISS}.
\end{proof}
}{}

We now have all the ingredients to prove Theorem~\ref{thm:PEHgISS}.

{\bf Proof of Theorem~\ref{thm:PEHgISS}: }
Using the same arguments as in the proof of Theorem~\ref{thm:PEGlobalStabilityHSg}, we conclude that, by \iftoggle{dissertation}{Assumptions~\ref{asm:PEBoundedPhicPHid} and~\ref{asm:PEHybridPEPsi}}{items 1 and 2 of Theorem~\ref{thm:PEGlobalStabilityHSg}}, the conditions of Theorem~\ref{thm:COMISS} are satisfied with $\mu_0$ and $a_M$ from Theorem~\ref{thm:PEGlobalStabilityHSg}.
It can be shown that, under Assumption~\ref{asm:PEboundedf}, the hybrid system that is obtained by expressing $\HS_\nu$ in error coordinates is equivalent to $\HS$ in \eqref{eqn:COMHS} with $A, B$ in \eqref{eqn:PEABd} and $d_c, d_d$ in \eqref{eqn:PEDcDd}.
Hence, it follows from Theorem~\ref{thm:COMISS} that, for each solution $\xi$ to $\HS_\nu$ from ${\cal X}_0$, the parameter estimation error $\tilde \theta = \theta - \hat \theta$ satisfies, for all $(t,j) \in \dom \xi$,
\begin{align} \label{eqn:PEtildethetaISS2}
    |\tilde \theta(t,j)| \leq \sqrt{\frac{p_M}{p_m}} \mathrm{e}^{-\omega(t+j)} |\tilde \theta(0,0)| + \rho \| d \|_{(t,j)},
\end{align}
with $(t,j) \mapsto d(t,j)$ as in \eqref{eqn:COMd} and $\rho$, $\omega$ from Theorem~\ref{thm:COMISS}, with $p_m, p_M$ substituted by $p_m, p_M$ from Theorem~\ref{thm:PEGlobalStabilityHSg}.

Using the definition of ${\cal A}_g$ in \eqref{eqn:PEcalAg}, we rewrite $|\xi|_{{\cal A}_g}$ for all $(t,j) \in \dom \xi$ as
%
    $|\xi(t,j)|_{{\cal A}_g} = \sqrt{|\tilde \theta(t,j)|^2 + |\varepsilon(t,j)|^2}$. 
%
Since, by Assumption~\ref{asm:PEboundedf}, the conditions of Lemma~\ref{lem:PEepsISS} are satisfied, we substitute the bounds in \eqref{eqn:PEepsISS} and \eqref{eqn:PEtildethetaISS2}. Using that, for any $\alpha, \beta \in \reals$, $\alpha \beta \leq \frac{1}{2}(\alpha^2 + \beta^2)$, we obtain
\iftoggle{long}{
\begin{align*} 
    |\xi(t,j)|_{{\cal A}_g} 
    &\leq
        \sqrt{\frac{2 p_M}{p_m}} \mathrm{e}^{- \min\{\omega,\lambda_\varepsilon\}(t+j)} |\xi(0,0)|_{{\cal A}_g} \\
        &\quad + \sqrt{2} \max\{\rho,\rho_\varepsilon\}  \sqrt{\| d \|_{(t,j)}^2 + \|d_\varepsilon\|_{(t,j)}^2}
\end{align*}
}{
    $|\xi(t,j)|_{{\cal A}_g} 
    \leq
    \sqrt{\frac{2 p_M}{p_m}} \mathrm{e}^{- \min\{\omega,\lambda_\varepsilon\}(t+j)} |\xi(0,0)|_{{\cal A}_g}
    + \sqrt{2} \max\{\rho,\rho_\varepsilon\}  \sqrt{\| d \|_{(t,j)}^2 + \|d_\varepsilon\|_{(t,j)}^2}$ 
}
for all $(t,j) \in \dom \xi$. Hence, \eqref{eqn:PEISSbound} holds.

To conclude the proof, we upper bound $d_c$, $d_d$, $\alpha_c$, and $\alpha_d$ for all $(t,j) \in \dom \xi$. The bounds for $\alpha_c$ and $\alpha_d$ in Theorem~\ref{thm:PEHgISS} follow directly from Assumption~\ref{asm:PEboundedf} and the definitions of $\alpha_c$ and $\alpha_d$ in \eqref{eqn:PEDcDd}. Moreover, since, by \iftoggle{dissertation}{Assumptions~\ref{asm:PEBoundedPhicPHid} and~\ref{asm:PEboundedf}}{item 1 of Theorem~\ref{thm:PEGlobalStabilityHSg} and Assumption~\ref{asm:PEboundedf}}, the conditions of Lemmas~\ref{prop:PEpsiStab} and \ref{lem:PEepsISS} are satisfied, we have from \eqref{eqn:PEDcDd} that, for each solution $\xi$ to ${\HS}_\nu$ from ${\cal X}_0$,
\begin{equation*}
\mbox{$\squeezespaces{0.4}
\begin{aligned}
|d_c(t,j)| 
&\leq \gamma_c |\psi(t,j)| ( |\varepsilon(t,j)| + |\nu(t,j)| ) \\
&\leq\gamma_c \psi_M ( \mathrm{e}^{-\lambda_\varepsilon (t + j)} |\varepsilon(0,0)| \\
&\qquad + (\rho_\varepsilon \max\{\lambda_c+L_c, 1-\lambda_d+L_d\} + 1) |\nu(t,j)| )
\end{aligned}
$}
\end{equation*}
for all $(t,j) \in \dom \xi$, with $\lambda_\varepsilon$ and $\rho_\varepsilon$ from Theorem~\ref{thm:PEHgISS}, and the last inequality follows from \eqref{eqn:PEepsISS} the definition of $d_\varepsilon$ in \eqref{eqn:PEdvarepsilon}. 
Next, using that
%
    $\frac{|\psi(t,j+1)|}{\gamma_d + |\psi(t,j+1)|^2} \leq \frac{1}{2 \sqrt{\gamma_d}}$
%
for all $(t,j) \in \Upsilon(\dom \xi)$, 
we have that, for all $(t,j) \in \Upsilon(\dom \xi)$,
\begin{align*}
    &|d_d(t,j)| 
    \leq \frac{1}{2 \sqrt{\gamma_d}} \big( \varepsilon(t,j+1)
    + \nu(t,j+1) \big) \\
    &\leq \frac{1}{2 \sqrt{\gamma_d}} ( \mathrm{e}^{-\lambda_\varepsilon (t + j + 1)} |\varepsilon(0,0)| \\
    &\quad\mbox{$\squeezespaces{0.5}+ (\rho_\varepsilon \max\{\lambda_c+L_c, 1-\lambda_d+L_d\} + 1) |\nu(t,j+1)| ) $} \tag*{$\qed$}
\end{align*}
\begin{remark}
    A similar ISS result as in Theorem~\ref{thm:PEHgISS} can be developed without Assumption~\ref{asm:PEboundedf} by constraining the range of the plant state $x$ and the input $u$ to a compact set.
    Under such conditions, it follows from the continuity of $f_c$ and $g_d$ that $d_c$, $d_d$, $\alpha_c$, $\alpha_d$ in \eqref{eqn:PEDcDd} can be upper bounded by functions of only $\nu$. Then, ISS follows from similar arguments as in the proof of Theorem~\ref{thm:PEHgISS}.
\end{remark}

\iftoggle{arxiv}{
\section{Case Studies}}
{\section{Spacecraft Bias Torque Estimation}}
\label{sec:PESimulationResults}
\iftoggle{arxiv}{
In this section, we present case studies that demonstrate the merits of our hybrid algorithm. Simulations are performed using the \href{https://hybrid.soe.ucsc.edu/sites/default/files/preprints/74.pdf}{Hybrid Equations Toolbox} \cite{74}.
\subsection{Motivational Example Revisited}

Recall the example in Section~\ref{sec:PEMotivation}, where the system \eqref{eqn:PEmotivationalplant} can be written as \eqref{eqn:PEplant} by setting $f_c$, $g_d$ in \eqref{eqn:PEplant} to zero. 
We employ $\HS_g$ in \eqref{eqn:PEHSg} to estimate $\theta$ in \eqref{eqn:PEmotivationalplant}. The algorithm is simulated for $\gamma_c = 1$, $\lambda_c = 0.1$, $\gamma_d = 1$, and $\lambda_d = 0.5$ alongside the continuous-time and discrete-time estimation algorithms from Section~\ref{sec:PEContDiscGrad} with the same parameters, where applicable. 
To illustrate the robustness of our algorithm, we also simulate $\HS_g$ with additive noise $(t,j) \mapsto \nu(t,j) = \sin(2t) [1 \, \, 1]^\top$ in the measurements of $x$.
Recall that the classical PE conditions \ref{C1} and \ref{C2} are not satisfied by $\phi_c$ and $\phi_d$ given above \eqref{eqn:PEmotivationalplant}. However, $\psi$ satisfies the hybrid PE condition \eqref{eqn:PEHybridPEg} with $\Delta = 2 \pi + 1$ and $\mu = 5.1$.

The simulation is performed from two separate initial conditions: one with $\varepsilon(0,0) = 0$ and one with $\varepsilon(0,0) \neq 0$. In particular, $x(0,0) = (3, 6)$, $\hat \theta(0,0) = (0,0)$, and
\begin{enumerate}[label=\arabic*.]
    \item
    $\mbox{$\squeezespaces{1} \psi(0,0) = 0, \; \eta(0,0) = -(3, 6) \implies \varepsilon(0,0) = (0,0) $}$
    \item 
    $
    \mbox{$\squeezespaces{0.65}
    \psi(0,0) = 0, \ \eta(0,0) = -(1.5, 3) \implies \varepsilon(0,0) = (1.5, 3)
    $}
    $
\end{enumerate}
producing the results in Figure~\ref{fig:PESim2}. 
When no noise is present, $|\xi|_{{\cal A}_g}$ converges exponentially to zero in accordance with Theorem~\ref{thm:PEGlobalStabilityHSg}, as shown in blue for the case with $\varepsilon(0,0) = 0$ and in green for the case with $\varepsilon(0,0) \neq 0$. When noise is present, $|\xi|_{{\cal A}_g}$ remains bounded in accordance with Theorem~\ref{thm:PEHgISS}, as shown in orange in Figure~\ref{fig:PESim2}.
\begin{figure}[!hbt]
    \centering
        \includegraphics[width=1\linewidth,height=\textheight,keepaspectratio]{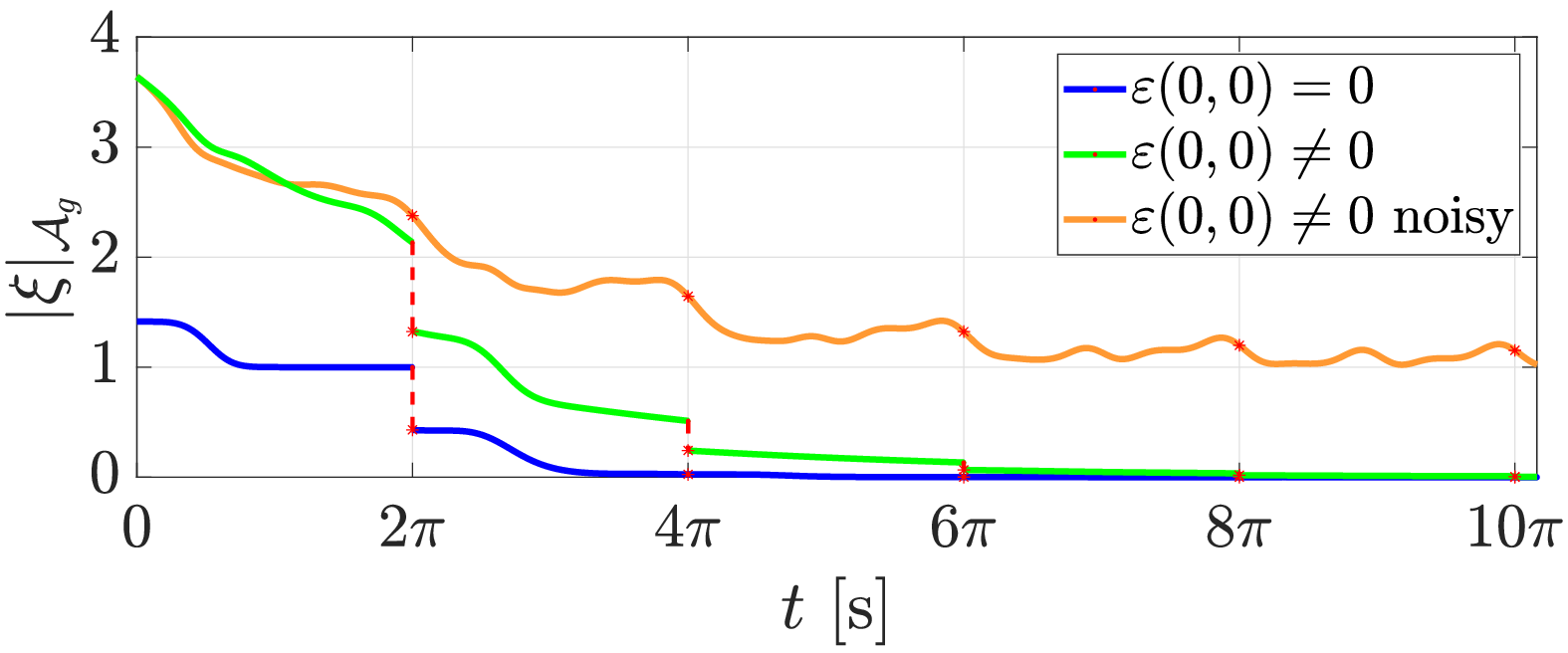}
        \caption{The projection onto $t$ of $|\xi|_{{\cal A}_g}$ for $\HS_g$.}
    \label{fig:PESim2}
\end{figure}

\subsection{Spacecraft Bias Torque Estimation}
}{
In this section, we present a case study that demonstrate the merits of our hybrid algorithm. Simulations are performed using the \href{https://hybrid.soe.ucsc.edu/sites/default/files/preprints/74.pdf}{Hybrid Equations Toolbox} \cite{74}.
}
%
Consider the problem of estimating a constant disturbance torque applied to a spacecraft, controlled by reaction wheels (RW) and reaction control system (RCS) thrusters. 
Such bias torques may arise in practice due to aerodynamic effects, gravity gradients, or solar radiation pressure differentials.
For simplicity, we consider the dynamics of a spacecraft rotating about only a single principle axis of inertia, although our approach can be extended to three-axis rotation.
In the following, we derive the closed-loop dynamics of the spacecraft when controlled by RW and RCS thrusters separately, and then combine the results into a single hybrid model.

%
%
%
The dynamics of a spacecraft rotating along a principle axis of inertia under the effect of RW are \cite{Sidi1997Spacecraft}
\begin{equation} \label{eqn:PESC1}
    J_s \ddot z = - J_w \dot \Omega + \theta,
\end{equation}
where $z \in \reals$ is the known pointing angle of the spacecraft, $\Omega \in \reals$ is the known rotational velocity of the RW, $J_s > 0$ is the known spacecraft moment of inertia, $J_w > 0$ is the known RW moment of inertia, and $\theta \in \reals$ is an unknown bias torque.

Suppose RW control the attitude to a pointing angle, $z_{\rm des} \in \reals$. 
The dynamics of the reaction wheel are \cite{Sidi1997Spacecraft}
\begin{equation} \label{eqn:PESC2}
    J_w \dot \Omega = \alpha(t),
\end{equation}
where $t \mapsto \alpha(t) \in \reals$ is the RW motor torque that is designed to maintain the spacecraft pointing angle.
Substituting \eqref{eqn:PESC2} into \eqref{eqn:PESC1}, we obtain
\begin{equation} \label{eqn:PESC3}
    J_s \ddot z = - \alpha(t) + \theta.
\end{equation}
When the bias torque is nonzero, the industry-standard proportional-derivative (PD) control scheme for the RW motor fails to yield zero pointing error in steady-state. In this case, a feedfoward term is added that compensates for the effect of the bias torque using an estimate of the bias, denoted by $\hat \theta$ \cite{Sidi1997Spacecraft}. Hence, the RW torque is
\begin{equation} \label{eqn:PEu}
    -\alpha(t) = K_P (z_{\rm des} - z(t)) - K_D \dot z(t) - \hat \theta(t),
\end{equation}
where $K_P, K_D > 0$ are design parameters.
From \eqref{eqn:PESC2}, \eqref{eqn:PESC3}, \eqref{eqn:PEu}, the dynamics of the closed-loop system are
\begin{align} \label{eqn:PESCflow}
    \ddot z &= \frac{- \alpha(t) + \theta}{J_s}, &
    \dot \Omega &= \frac{\alpha(t)}{J_w}.
\end{align}
%
%
%
The spacecraft pointing angle can be maintained only if an equivalent RW torque is delivered to counteract the bias torque. If the bias torque is nonzero, the angular velocity of the RW constantly increases in order to counteract the disturbance and the RW motor eventually reaches its maximum angular velocity. 
In order to avoid the RW motor from becoming saturated, ``momentum dumping'' is applied to decrease the angular velocity of the RW \cite{Sidi1997Spacecraft}. This procedure involves firing the RCS thrusters to generate a torque that is compensated by the attitude controller by actions that cause the RW to reduce their angular momentum.

The dynamics of a spacecraft rotating along a principle axis of inertia under the effect of RCS thrusters are \cite{Sidi1997Spacecraft}
\begin{equation} \label{eqn:PESC5}
    J_s \ddot z = M + \theta,
\end{equation}
where $M \in \reals$ is the known RCS thruster torque. For simplicity, we assume that the velocity of the RW is constant for the duration of each thruster firing. As a result, the RW dynamics do not play a role in \eqref{eqn:PESC5}.

Suppose that, at time $t \geq 0$, the thrusters are fired for $\delta > 0$ seconds. Integrating \eqref{eqn:PESC5} over the time interval $[t, t+\delta]$ yields 
$\dot z(t + \delta) = \dot z(t) + \frac{\delta}{J_s} (M + \theta)$.
If the thruster firing duration $\delta$ is negligibly small compared to the other time scales of the system, which is appropriate due to the slow spacecraft attitude maneuvering, we model the thruster firing as an instantaneous jump in the angular velocity of the spacecraft, given by
%
\begin{align} \label{eqn:PESCjump}
    \dot z^+ &= \dot z + \frac{\delta}{J_s}(M + \theta).
\end{align}
%
%
To avoid chatter, a timer, denoted by $\tau_s$, is used to briefly inhibit the RCS thrusters after each thruster firing. Each time the thrusters are fired, the timer is reset to zero.
%

By combining the expression in \eqref{eqn:PESCflow} and \eqref{eqn:PESCjump}, we express the closed-loop dynamics of the spacecraft as a hybrid system as in \eqref{eqn:PEplant}.
Given an input $u := (z_{\rm des}, \hat \theta)$, where $z_{\rm des} \in \reals$ is the desired constant spacecraft pointing angle and $\hat \theta \in \reals$ is an estimate of the unknown bias torque, the hybrid model of the spacecraft has state $x = (z, \dot z, \Omega, \tau_s) \in \reals^4$ and data
\begin{equation*} 
\allowdisplaybreaks
\begin{aligned}
    f_c(x,u(t,j)) &:=
    \matt{\dot z \\
          - \frac{1}{J_s} \alpha(x,u(t,j))  \\
          \frac{1}{J_w} \alpha(x,u(t,j))\\
          1}, &
    \!\!\!\! \phi_c(t,j) &:=
    \matt{0 \\ \frac{1}{J_s} \\ 0 \\ 0} \\
    g_d(x,u(t,j)) &:=
    \matt{z\\ 
            \dot z + \frac{\delta}{J_s} M \\ 
            \Omega \\
            0}, &
    \!\!\!\! \phi_d(t,j) &:=
    \matt{0 \\ \frac{\delta}{J_s} \\ 0 \\ 0}
\end{aligned}
\end{equation*}
where $\alpha(x,u(t,j)) := - K_P (z_{\rm des} - z) + K_D \dot z + \hat \theta(t,j)$.
The flow and jump sets of the hybrid spacecraft model implement the momentum dumping procedure. The system jumps each time the angular velocity of the RW exceeds a design parameter $\Omega_{\max}  > 0$ and the timer $\tau_s$ exceeds a design parameter $\tau^* > 0$, and flows otherwise, as
\begin{align*}
    C_P &:= \{ x \in \reals^4 : \Omega \leq \Omega_{\max} \} \cup \{ x \in \reals^4 : \tau_s \leq \tau^* \} \\
    D_P &:= \{ x \in \reals^4 : \Omega \geq \Omega_{\max}, \; \tau_s \geq \tau^*\}.
\end{align*}
We employ $\HS_g$ to estimate the unknown bias torque.
The closed-loop system is simulated\footnote{Code at \href{https://github.com/HybridSystemsLab/HybridGD_SpacecraftBiasTorque}{\url{https://github.com/HybridSystemsLab/HybridGD_SpacecraftBiasTorque}}} with initial conditions $x(0,0) = (0,0,0,0)$, $\hat \theta(0,0) = 0$, $\psi(0,0) = 0$, and $\eta(0,0) = -x(0,0)$.
The hybrid spacecraft model has parameters
$z_{\rm des} = 0$ rad, $\Omega_{\max} = 10000$ RPM, $J_s = 5000$ kg-m$^2$, $J_w = 0.1$ kg-m$^2$, $M = -10$ N-m, $\delta = 9.5$ sec, $\tau^* = 10$ sec, $K_p = 10$, $K_d = 1200$, and with an unknown bias torque of $\theta = 0.005$ N-m.
Our algorithm $\HS_g$ has parameters $\gamma_c = 0.0012$, $\lambda_c = 0.001$, $\gamma_d = 0.01$, and $\lambda_d = 0.5$.
With the initial conditions and design parameters given above, it can be shown numerically that the conditions of Theorem \ref{thm:PEGlobalStabilityHSg} hold.

The bias torque estimation error from $\HS_g$ converges exponentially to zero in accordance with Theorem \ref{thm:PEGlobalStabilityHSg}, as shown in Figure \ref{fig:PESpacecraftEst}. 
\begin{figure}[!hbt]
    \centering
        \includegraphics[width=1\linewidth,height=\textheight,keepaspectratio]{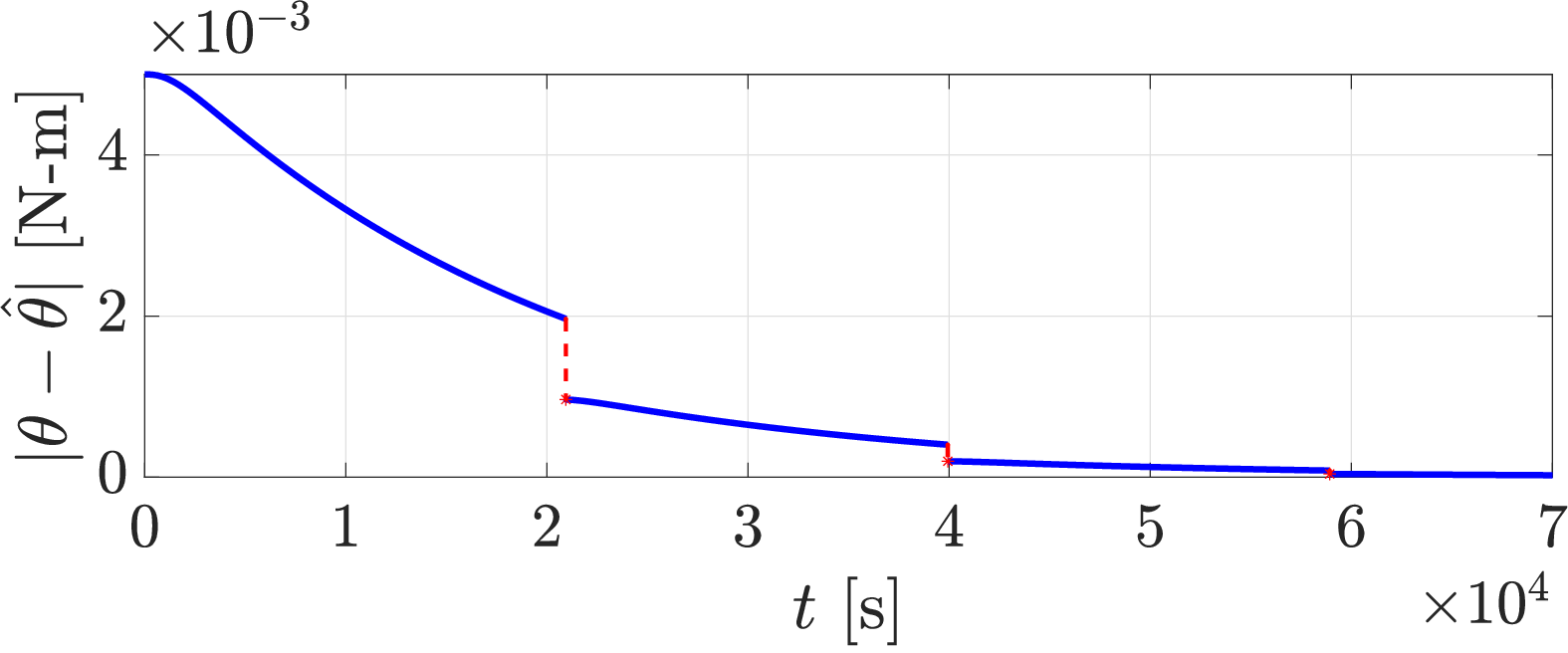}
        \caption{The projection onto $t$ of the bias torque estimation error for $\HS_g$.}
    \label{fig:PESpacecraftEst}
\end{figure}
The spacecraft pointing angle error and RW angular velocity are shown in the top and bottom plots, respectively, in Figure \ref{fig:PESpacecraftZ}, where the control performance resulting from our hybrid algorithm is compared against an industry-standard PID control scheme that is tuned to achieve a similar pointing error convergence rate during flows. For the PID controller, we inhibit accumulation of the integrator during each thruster firing, otherwise the spacecraft pointing angle fails to converge to the set point.
With the exception of the transients caused by the thruster firings, the pointing error converges to zero for both controllers. However, our hybrid algorithm converges faster due to our estimator's ability to leverage information during both flows and jumps to estimate the unknown bias torque. 
\begin{figure}[!hbt]
    \centering
        \includegraphics[width=1\linewidth,height=\textheight,keepaspectratio]{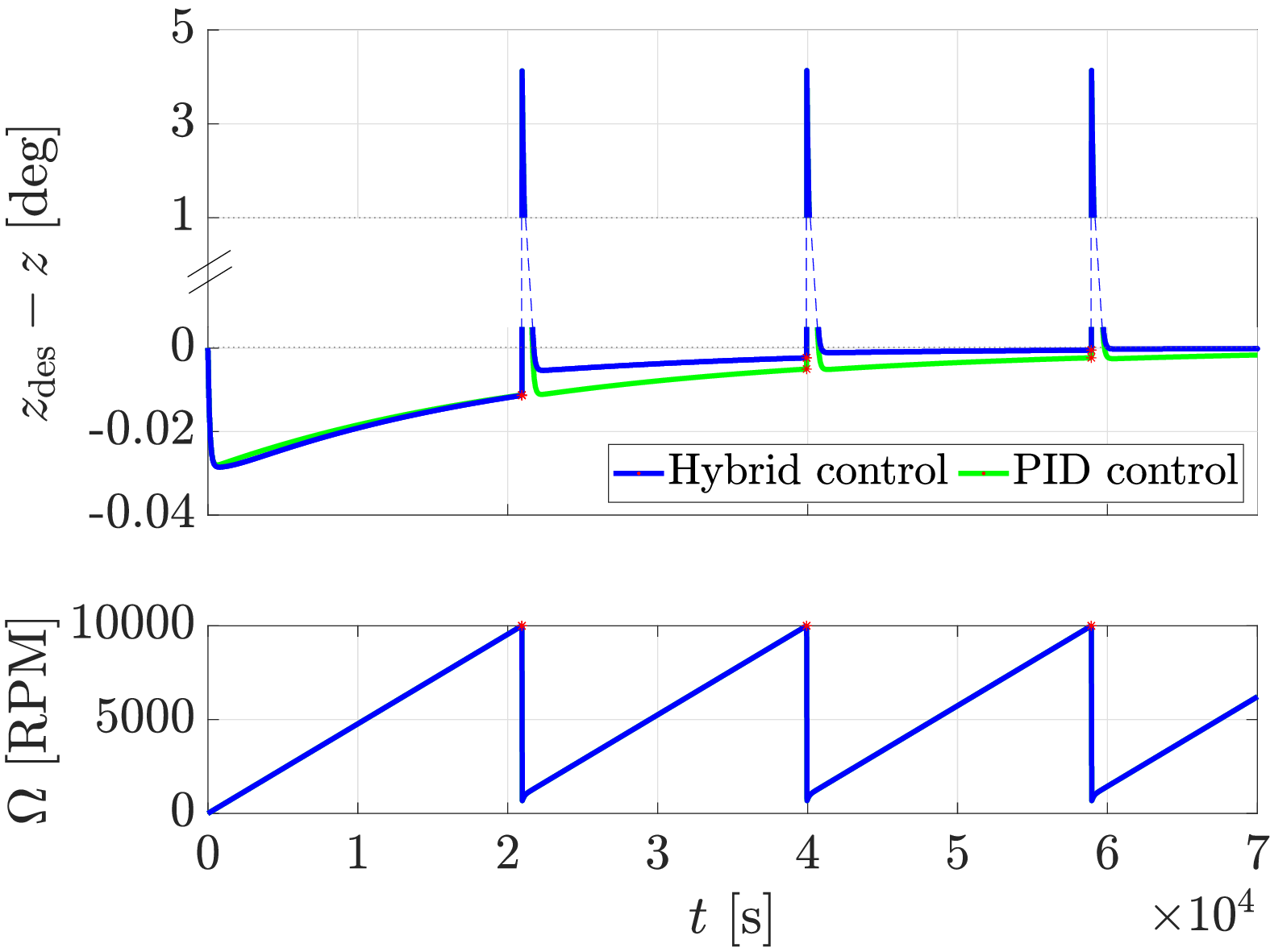}
        \caption{The projection onto $t$ of the spacecraft pointing angle error (top) and the RW angular velocity (bottom).}
    \label{fig:PESpacecraftZ}
\end{figure}

\section{Conclusion}
\label{sec:PEConclusion}
In this paper, we propose a hybrid algorithm for estimating unknown parameters in a class of hybrid systems with nonlinear dynamics that are affine in the unknown parameter. We show that our algorithm guarantees exponential convergences of the parameter estimate to the true value under a notion of hybrid persistence of excitation that relaxes the classical continuous-time and discrete-time persistence of excitation conditions.
Moreover, we show that the parameter estimate is ISS with respect to hybrid noise in the measurements of the plant state.
To demonstrate its practicality, we apply $\HS_g$ to estimate an unknown bias torque applied to a simplified model of a spacecraft controlled by reaction wheels and reaction control thrusters.
Future work on this topic includes extending our proposed algorithm to estimate the unknown parameters for hybrid dynamical systems with unknown jump times.
%
%

\begin{ack}                               
Research by R. S. Johnson and R. G. Sanfelice has been partially supported by NSF Grants no. ECS-1710621, CNS-2039054, and CNS-2111688, by AFOSR Grants no. FA9550-19-1-0053, FA9550-19-1-0169, and FA9550-20-1-0238, and by ARO Grant no. W911NF-20-1-0253.

We thank A. Saoud, M. Maghenem, and A. Loria for their insightful discussions on this topic.
\end{ack}

\bibliographystyle{plain}        
\bibliography{arxiv}           

\renewcommand*\appendixpagename{\normalsize Appendices}
\appendix
\appendixpage
\renewcommand{\thehelptheorem}{\Alph{section}.\arabic{helptheorem}}
\vspace{-1mm}
\section{Constants $\kappa_g$ and $\lambda_g$ in Theorem~\ref{thm:PEGlobalStabilityHSg}} \label{apx:PEDefKappagLambdag}
\vspace{-1mm}
Given $\phi_c, \phi_d : E \to \reals^{n \times p}$, $\gamma_c, \lambda_c, \gamma_d, \phi_M, \Delta, \mu > 0$, $\lambda_d \in (0, 2)$, $\psi_0 \geq 0$, $q_M \geq q_m > 0$, and $\zeta \in (0, 1)$ from Theorem~\ref{thm:PEGlobalStabilityHSg}, suitable choices of $\kappa_g$ and $\lambda_g$ in Theorem~\ref{thm:PEGlobalStabilityHSg} are 
$\kappa_g := \sqrt{3} \kappa$ and $\lambda_g := \min\{\lambda, b\}$, where
\vspace{-1mm}
%
%
%
\begin{align*} 
        \kappa &:= 
        2 \max \left\{ 
        \dfrac{p_M}{p_m}, \,
        a \rho \sqrt{\dfrac{p_M}{p_m}}
        \right\}, \quad
        \lambda := 
        \frac{1}{2} \min \left\{ \omega, b \right\}, \\
        a &:= \max \left\{ \gamma_c \psi_M, \frac{1}{2 \sqrt{\gamma_d}} \right\}, \\
        b &:= \frac{1}{2} \min \left\{ 
        \lambda_c, \; - \ln\left( 1 - \lambda_d ( 2 - \lambda_d ) \right) \right\}, \\
        p_m &:= q_m, \quad \
        p_M := q_m + \frac{q_M \kappa_0^2}{2 \lambda_0} + \frac{q_M \kappa_0^2 \mathrm{e}^{2 \lambda_0}}{\mathrm{e}^{2 \lambda_0} - 1}, \\
        \rho &:= \sqrt{\frac{2 p_M^3}{q_m p_m \zeta} \left( \frac{2 p_M}{q_m} + 1 \right)}, \\
        \omega &:= 
        \mbox{$\squeezespaces{0.8}
        \dfrac{1}{2} \min \left\{ \dfrac{q_m}{2 p_M} (1 - \zeta), \,
        -\ln \left( 1 - \dfrac{q_m}{2 p_M} (1 - \zeta) \right) \right\}
        $}, \\
        \kappa_0 &:= \sqrt{\frac{1}{1 - \sigma}}, \qquad\quad
        \lambda_0 := - \frac{\ln (1 - \sigma)}{2 (\Delta + 1)}, \\
        \sigma &:= \frac{2 \mu_0}{\big( 1 + \sqrt{(a_M + 2) (\Delta + 2)^3 (a_M (\Delta + 2) + 1/2)} \big)^{2}}, \\
        \mu_0 &:= \min \left\{ \gamma_c, \frac{1}{2(\gamma_d + \psi_M^2)} \right\} \mu, \qquad \quad 
        a_M := \gamma_c \psi_M^2, \\
        \psi_M &:= \psi_0 + \max \left\{
        \frac{1}{\lambda_c}, 
        \frac{\sqrt{2 \lambda_d (2 - \lambda_d) + 16}}{\lambda_d (2 - \lambda_d)} \right\} \phi_M.
\end{align*}
\vspace{-2mm}
\section{Proof of Theorem \ref{thm:COMISS}} \label{apx:ISS}
\vspace{-1mm}
To prove Theorem \ref{thm:COMISS}, we first require some auxiliary results for the hybrid system $\HS$ when the disturbances $d_c$ and $d_d$ are equal to zero.
%
We denote this system as $\HS_0$, with state $\xi = (\vartheta,\tau,k) \in {\cal X}$ and dynamics
\iftoggle{long}{}{\vspace{-1mm}}
\begin{equation} \label{eqn:PEHS0}
    \HS_0 :
    \left\{
    \begin{aligned}
        \dot \xi &=
        \begin{bmatrix}
        -A(\tau,k) \vartheta \\
        1 \\
        0
        \end{bmatrix} =: F_0(\xi)
        & \xi &\in C_0  \\
        \xi^+ &= 
        \begin{bmatrix}
        \vartheta - B(\tau,k) \vartheta \\
        \tau \\
        k + 1
        \end{bmatrix} =: G_0(\xi)
        & \xi &\in D_0
    \end{aligned}
    \right.
\end{equation}
where $C_0 := C$ and $D_0 := D$, with $C$ and $D$ below \eqref{eqn:COMHS}. 
%

%
Inspired by \cite{242}, we establish sufficient conditions that ensure the hybrid system $\HS_0$ induces global pre-exponential stability of the set ${\cal A}$ in \eqref{eqn:PEcalA}.
\begin{theorem} \label{thm:PEMain}
Given the hybrid system $\HS_0$ in \eqref{eqn:PEHS0}, suppose that Assumptions~\ref{asm:COMABStructure} and \ref{asm:COMHybridPE} hold. Then, each solution $\xi$ to $\HS_0$ satisfies
\begin{equation} \label{eqn:PEExpStabilityA}
\begin{aligned} 
    |\xi(s,i)|_{\cal A} \leq \kappa_0 \mathrm{e}^{-\lambda_0(s + i - t - j)} |\xi(t,j)|_{\cal A}
\end{aligned}
\end{equation}
for all $(s,i),(t,j) \in \dom \xi$ satisfying $s + i \geq t + j$, where
%
%
%
\begin{align*}
    \kappa_0 &:= \sqrt{\frac{1}{1 - \sigma}}, \qquad \qquad
    \lambda_0 := - \frac{\ln (1 - \sigma)}{2 (\Delta + 1)},\\
    \sigma &:= \frac{2 \mu_0}{\big( 1 + \sqrt{(a_M + 2) (\Delta + 2)^3 (a_M (\Delta + 2) + 1/2)} \big)^{2}}
\end{align*}
%
%
%
%
with $a_M$, $\mu_0, \Delta$ from Assumptions \ref{asm:COMABStructure} and \ref{asm:COMHybridPE}.
\end{theorem}
\vspace{-3mm}
\begin{proof}
The proof of Theorem~\ref{thm:PEMain} follows along the same lines as the proof of \cite[Theorem~1]{242}.
\vspace{-3mm}
\end{proof}
Next, we recall the following result from \cite{Horn2012Matrix}.
\begin{lemma} \label{lem:PEInvertibleMatrix}
Given $B \in \reals^{p \times p}$, if $|B| < 1$, then $I - B$ is invertible.
\vspace{-2mm}
\end{lemma}
Finally, we establish the following lemma.
\begin{lemma} \label{lem:PEP}
    Given the hybrid system $\HS_0$ in \eqref{eqn:PEHS0}, suppose that Assumptions~\ref{asm:COMABStructure} and \ref{asm:COMHybridPE} hold and let the hybrid time domain $E$ come from these assumptions. 
    Then, for each $q_M \geq q_m > 0$ and each symmetric matrix function $Q: E \to \reals^{p \times p}$ satisfying 
    \iftoggle{arxiv}{
    \begin{align} \label{eqn:PEQbound}
        q_m I \leq Q(t,j) \leq q_M I
        \quad \forall (t,j) \in E,
    \end{align}
    }{
    $q_m I \leq Q(t,j) \leq q_M I$ for all $(t,j) \in E$, 
    }
    there exists a symmetric matrix function $P : E \rightarrow \reals^{p \times p}$ satisfying
    \begin{align} \label{eqn:PEPbound}
        p_m I \leq P(t,j) \leq p_M I
        \quad \forall (t,j) \in E,
    \end{align}
    where
    \iftoggle{arxiv}{
    \begin{align} \label{eqn:PEpm_pM}
        p_m &:= q_m, &
        p_M &:= q_m + \frac{q_M \kappa_0^2}{2 \lambda_0} + \frac{q_M \kappa_0^2 \mathrm{e}^{2 \lambda_0}}{\mathrm{e}^{2 \lambda_0} - 1},
    \end{align}
    }{
            $p_m := q_m$, 
            $p_M := q_m + \frac{q_M \kappa_0^2}{2 \lambda_0} + \frac{q_M \kappa_0^2 \mathrm{e}^{2 \lambda_0}}{\mathrm{e}^{2 \lambda_0} - 1}$, 
    }
    with $\kappa_0$ and $\lambda_0$ from Theorem~\ref{thm:PEMain}.
    Moreover,
    for each $j \in \nats$ and \iftoggle{arxiv}{for }{}almost all $\mbox{$\squeezespaces{0.5}t \in I^j := \{t : (t,j) \in E\}$}$,
    $(t,j) \mapsto P(t,j)$ satisfies
    \begin{equation} \label{eqn:PEPdot}
    \hspace{-2mm}
    \mbox{$\squeezespaces{0.15}
    \begin{aligned} 
        \frac{d}{dt} P (t,j) - P(t,j) A(t,j) - A(t,j)^\top P(t,j) &\leq - Q(t,j)
    \end{aligned}
    $}
    \hspace{-3mm}
    \end{equation}
    and, for all $(t,j) \in \Upsilon(E)$, with $\Upsilon$ as in \eqref{eqn:COMUpsilon},
    \begin{align} \label{eqn:PEPplus}
            (I - B(t,j))^\top P(t,j+1) (I - B(t,j)) &- P(t,j) \\ 
            &\leq - Q(t,j). \notag
    \end{align}
\end{lemma}
\iftoggle{arxiv}{
\begin{proof}
Let $U : E \to \reals^{p \times p}$ be such that $U(0,0)$ is invertible and, for each $j \in \nats$ and almost all $t \in I^j$,
\begin{equation} \label{eqn:PEUdynamicsFlow}
\begin{aligned}
    \frac{d}{dt} U(t,j) &= - A(t,j) U(t,j)
\end{aligned}
\end{equation}
and, for all $(t,j) \in \Upsilon(E)$, with $\Upsilon$ as in \eqref{eqn:COMUpsilon},
\begin{equation} \label{eqn:PEUdynamicsJump}
\begin{aligned}
    U(t,j+1) &= U(t,j) - B(t,j) U(t,j).
\end{aligned}
\end{equation}
Then, for all $(t,j),(t',j') \in E$, we define
\begin{align} \label{eqn:PEPhiDefn}
    \Phi(t,j,t',j') := U(t,j) U(t',j')^{-1},
\end{align}
where, in view of Lemma~\ref{lem:PEInvertibleMatrix}, $U(t,j)$ is invertible for all $(t,j) \in E$ since $U(0,0)$ is invertible and, by Assumption~\ref{asm:COMABStructure}, $|B(t,j)| < 1$ for all $(t,j) \in \Upsilon(E)$.

By the equivalence between the dynamics of $U$ and the $\vartheta$ component of $\xi$ in \eqref{eqn:PEHS0}, we have that, for each solution $\xi$ to $\HS_0$ and each $(t,j),(t',j') \in \dom \xi$,\footnote{
Since each solution $\xi$ to $\HS_0$ inherits the hybrid time domain $E$, it follows that $\dom \xi = E$, and thus $\Phi(t,j,t',j')$ is well defined for all $(t,j),(t',j') \in \dom \xi$,
}
\begin{align} \label{eqn:PEPhiStateTransition}
    \vartheta(t,j) = \Phi(t,j,t',j') \vartheta(t',j').
\end{align}
Hence, $\Phi$ is the state transition matrix for $\vartheta$. Note that $\Phi$ is not necessarily smooth at jumps.

Next, we define $(t,j) \mapsto P(t,j)$ as
\begin{equation} \label{eqn:PEP}
\begin{aligned} 
    P(t,j) := P_c(t,j) + P_d(t,j) + q_m I
\end{aligned}
\end{equation}
for all $(t,j) \in E$, with
\begin{align*}
    P_c(t,j) 
    &:= \sum_{i = j}^{J} \int_{\max\{t,t_i\}}^{t_{i+1}} \!\!\!\! \Phi(s,i,t,j)^\top Q(s,i) \Phi(s,i,t,j) d s \\
    P_d(t,j) 
    &:= \sum_{i = j}^{J} \Phi(t_{i+1},i,t,j)^\top Q(t_{i+1},i) \Phi(t_{i+1},i,t,j),
\end{align*}
where $t_{J+1} := T$, with $J := \sup_j E$ and $T := \sup_t E$. Note that the term $q_m I$ in \eqref{eqn:PEP} was chosen for simplicity -- any positive definite matrix would suffice.

We first show that \eqref{eqn:PEPbound} holds. 
Since, for all $(t,j) \in E$, $P_c(t,j) \geq 0$ and $P_d(t,j) \geq 0$, a lower bound on $P$ in \eqref{eqn:PEP} is $P(t,j) \geq q_m I$ for all $(t,j) \in E$.
Next, we develop an upper bound on $P$. 
Since, by Assumptions~\ref{asm:COMABStructure} and \ref{asm:COMHybridPE}, the conditions of Theorem \ref{thm:PEMain} are satisfied, it follows from \eqref{eqn:PEExpStabilityA} and from the equivalence between $|\xi|_{\cal A}$ and $|\vartheta|$ that, for each solution $\xi = (\vartheta,\tau,k)$ to $\HS_0$ and each $(s,i),(t,j) \in \dom \xi$ satisfying $s \geq t$ and $i \geq j$, 
%
    $|\vartheta(s,i)| \leq \kappa_0 \mathrm{e}^{-\lambda_0(s + i - t - j)} |\vartheta(t,j)|$
%
with $\kappa_0$ and $\lambda_0$ from Theorem \ref{thm:PEMain}.
By substituting \eqref{eqn:PEPhiStateTransition} into the expression above, we have that, for each $(s,i),(t,j) \in \dom \xi$ satisfying $s \geq t$ and $i \geq j$,
%
    $|\Phi(s,i,t,j) \vartheta(t,j)|  
    \leq \kappa_0 \mathrm{e}^{-\lambda_0(s + i - t - j)} |\vartheta(t,j)|$
%
which, if $|\vartheta(t,j)| \neq 0$, implies that
%
    $|\Phi(s,i,t,j) \vartheta(t,j)| / |\vartheta(t,j)|  
    \leq \kappa_0 \mathrm{e}^{-\lambda_0(s + i - t - j)}$. 
%
Since this inequality holds for any $\vartheta(t,j) \in \reals^p \setminus \{0\}$, it follows from the equivalence between $\dom \xi$ and $E$ that, for each $(s,i),(t,j) \in E$ satisfying $s \geq t$ and $i \geq j$,
\begin{align*}
    |\Phi(s,i,t,j)|
    &= \sup_{r \in \reals^p \setminus \{0\}} 
    \frac{|\Phi(s,i,t,j) r|}{|r|}
    \leq \kappa_0 \mathrm{e}^{-\lambda_0(s + i - t - j)}.
\end{align*}
Then, from the definitions of $P_c$ and $P_d$ below \eqref{eqn:PEP},
\begin{align*}
    P_c(t,j) 
    %
    %
    &\leq q_M \int_{0}^{\infty} \kappa_0^2 \mathrm{e}^{-2 \lambda_0 s} d s I
    = \frac{q_M \kappa_0^2}{2 \lambda_0} I
\end{align*}
and
\begin{align*}
    P_d(t,j) 
    &\leq q_M \sum_{i = 0}^{\infty} \kappa_0^2 \mathrm{e}^{- 2 \lambda_0 i} I
    = \frac{q_M \kappa_0^2 \mathrm{e}^{2 \lambda_0}}{\mathrm{e}^{2 \lambda_0} - 1} I.
\end{align*}
From the bounds above and the definition of $P$ in \eqref{eqn:PEP}, we conclude that \eqref{eqn:PEPbound} holds with $p_m$, $p_M$ in \eqref{eqn:PEpm_pM}.

Next, we show that \eqref{eqn:PEPdot} holds. We differentiate $P$ during flows and use that, for each $(s,i) \in E$ and each $j \in \nats$ and for almost all $t \in I^j$,
\begin{align*}
    \frac{d}{dt} \Phi(s,i,t,j) 
    = \Phi(s,i,t,j) A(t,j).
\end{align*}
This property follows from \eqref{eqn:PEUdynamicsFlow} and from the definition of $\Phi$ in \eqref{eqn:PEPhiDefn}.
For readability, we define
\begin{align} \label{eqn:PEPiDefinition}
    \Pi(s,i,t,j) := \Phi(s,i,t,j)^\top Q(s,i) \Phi(s,i,t,j).
\end{align}
Using the Leibniz integral rule, we obtain that, for each $j \in \nats$ and for almost all $t \in I^j$,
\begin{align} \label{eqn:PEP_cFlow}
    &\frac{d}{dt} P_c(t,j) \notag =
    A(t,j)^\top \bigg(
    \sum_{i = j}^{J} \int_{\max\{t, t_i\}}^{t_{i+1}} \!\!\!\! \Pi(s,i,t,j) d s \bigg) \notag \\
    &\quad + \bigg( \sum_{i = j}^{J} \int_{\max\{t, t_i\}}^{t_{i+1}} \!\!\!\! \Pi(s,i,t,j) d s \bigg)
    A(t,j) - Q(t,j) \notag \\
    &= A(t,j)^\top P_c(t,j) + P_c(t,j) A(t,j) - Q(t,j)
\end{align}
and 
\begin{align} \label{eqn:PEP_dFlow}
    \frac{d}{dt} P_d(t,j)
    &= A(t,j)^\top \bigg( 
    \sum_{i = j}^{J} \Pi(t_{i+1},i,t,j)
    \bigg) \notag \\
    &\quad + \bigg( 
    \sum_{i = j}^{J} \Pi(t_{i+1},i,t,j)
    \bigg) A(t,j) \notag \\
    &= A(t,j)^\top P_d(t,j) + P_d(t,j) A(t,j).
\end{align}
Combining the expressions in \eqref{eqn:PEP_cFlow} and \eqref{eqn:PEP_dFlow}, and using the definition of $P$ in \eqref{eqn:PEP}, we have that, for each $j \in \nats$ and for almost all $t \in I^j$,
\begin{align*}
    \frac{d}{dt} P(t,j) 
    &= A(t,j)^\top P(t,j) + P(t,j) A(t,j) \\
    &\quad - Q(t,j) - 2 q_m A(t,j) \\
    &\leq A(t,j)^\top P(t,j) + P(t,j) A(t,j) - Q(t,j).
\end{align*}
The inequality follows from the fact that, by Assumption~\ref{asm:COMABStructure}, $A(t,j) \geq 0$ for all $(t,j) \in E$. 
Hence, \eqref{eqn:PEPdot} holds.

To conclude the proof, we show that \eqref{eqn:PEPplus} holds. We use the property that, for each $(t,j),(s,i) \in \Upsilon(E)$,
\begin{align*}
    &\Phi(s,i,t,j+1)(I - B(t,j)) 
    = \Phi(s,i,t,j).
\end{align*}
This property follows from \eqref{eqn:PEUdynamicsJump} and from the definition of $\Phi$ in \eqref{eqn:PEPhiDefn}.
Then, for each $(t,j) \in \Upsilon(E)$,
\begin{align*}
    &(I - B(t,j))^\top P_c(t,j+1) (I - B(t,j)) \\
    %
    &\quad = \sum_{i = j+1}^{J} \int_{\max\{t, t_i\}}^{t_{i+1}} \!\!\!\! \Pi(s,i,t,j) d s.
\end{align*}
Since the value of ordinary time $t$ is the same immediately before and after each jump, it follows that, for each $(t,j) \in \Upsilon(E)$, $\max\{t, t_j\} = t = t_{j+1}$. Hence, we rewrite the expression above as
\begin{align} \label{eqn:PEP_cJump}
    &(I - B(t,j))^\top P_c(t,j+1) (I - B(t,j)) \notag \\
    %
    %
    %
    &\quad = \sum_{i = j}^{J} \int_{\max\{t, t_i\}}^{t_{i+1}} \!\!\!\! \Pi(s,i,t,j) d s = P_c(t,j).
\end{align}
Focusing now on $P_d$, for each $(t,j) \in \Upsilon(E)$,
\begin{align*}
    &(I - B(t,j))^\top P_d(t,j+1) (I - B(t,j)) - P_d(t,j) \\
    &\quad= - \Pi(t_{j+1},j,t,j).
\end{align*}
From \eqref{eqn:PEPiDefinition} and the fact that $t = t_{j+1}$ at each jump,
\begin{equation} \label{eqn:PEP_dJump}
\hspace{-2mm}
\mbox{$\squeezespaces{0.7}
\begin{aligned}
    &(I - B(t,j))^\top P_d(t,j+1) (I - B(t,j)) - P_d(t,j) \\
    &\quad = - \Phi^\top(t,j,t,j) Q(t,j) \Phi(t,j,t,j)
    = - Q(t,j).
\end{aligned}
$}
\hspace{-2mm}
\end{equation}
Using the definition of $P$ in \eqref{eqn:PEP}, it follows from \eqref{eqn:PEP_cJump} and \eqref{eqn:PEP_dJump} that, for each $(t,j) \in \Upsilon(E)$,
\begin{equation*}
\begin{aligned}
    &(I - B(t,j))^\top P(t,j+1) (I - B(t,j)) - P(t,j) \\
    &\quad = - Q(t,j) - q_m B(t,j) ( 2 I - B(t,j) )
    \leq -Q(t,j),
\end{aligned}
\end{equation*}
where the inequality holds since, by Assumption~\ref{asm:COMABStructure}, $B(t,j) \geq 0$ and $|B(t,j)| < 1$ for all $(t,j) \in E$. 
%
%
%
\end{proof}
}{
\vspace{-8mm}
\begin{proof}
Let $U : E \to \reals^{p \times p}$ be such that $U(0,0)$ is invertible and, for each $j \in \nats$ and almost all $t \in I^j$,
\begin{equation} \label{eqn:PEUdynamicsFlow}
\vspace{-1mm}
\begin{aligned}
    \frac{d}{dt} U(t,j) &= - A(t,j) U(t,j)
\end{aligned}
\vspace{-1mm}
\end{equation}
and, for all $(t,j) \in \Upsilon(E)$, with $\Upsilon$ as in \eqref{eqn:COMUpsilon},
\begin{equation} \label{eqn:PEUdynamicsJump}
\vspace{-1mm}
\begin{aligned}
    U(t,j+1) &= U(t,j) - B(t,j) U(t,j).
\end{aligned}
\vspace{-1mm}
\end{equation}
Then, for all $(t,j),(t',j') \in E$, we define
\begin{equation} \label{eqn:PEPhiDefn}
\vspace{-1mm}
\begin{aligned}
    \Phi(t,j,t',j') := U(t,j) U(t',j')^{-1},
\end{aligned}
\vspace{-1mm}
\end{equation}
where, in view of Lemma~\ref{lem:PEInvertibleMatrix}, $U(t,j)$ is invertible for all $(t,j) \in E$ since $U(0,0)$ is invertible and, by Assumption~\ref{asm:COMABStructure}, $|B(t,j)| < 1$ for all $(t,j) \in \Upsilon(E)$.

By the equivalence between the dynamics of $U$ and the $\vartheta$ component of $\xi$ in \eqref{eqn:PEHS0}, we have that, for each solution $\xi$ to $\HS_0$ and each $(t,j),(t',j') \in \dom \xi$,\footnote{
Since each solution $\xi$ to $\HS_0$ inherits the hybrid time domain $E$, it follows that $\dom \xi = E$, and thus $\Phi(t,j,t',j')$ is well defined for all $(t,j),(t',j') \in \dom \xi$,
}
\begin{align} \label{eqn:PEPhiStateTransition}
    \vartheta(t,j) = \Phi(t,j,t',j') \vartheta(t',j').
\end{align}
Hence, $\Phi$ is the state transition matrix for $\vartheta$. Note that $\Phi$ is not necessarily smooth at jumps.

Next, we define $(t,j) \mapsto P(t,j)$ as
\begin{equation} \label{eqn:PEP}
\begin{aligned} 
    P(t,j) := P_c(t,j) + P_d(t,j) + q_m I
\end{aligned}
\end{equation}
for all $(t,j) \in E$, with
\begin{align*}
    P_c(t,j) 
    &:= \sum_{i = j}^{J} \int_{\max\{t,t_i\}}^{t_{i+1}} \!\!\!\! \Phi(s,i,t,j)^\top Q(s,i) \Phi(s,i,t,j) d s \\
    P_d(t,j) 
    &:= \sum_{i = j}^{J} \Phi(t_{i+1},i,t,j)^\top Q(t_{i+1},i) \Phi(t_{i+1},i,t,j),
\end{align*}
where $t_{J+1} := T$, with $J := \sup_j E$ and $T := \sup_t E$. Note that the term $q_m I$ in \eqref{eqn:PEP} was chosen for simplicity -- any positive definite matrix would suffice.

We first show that \eqref{eqn:PEPbound} holds. 
Since, for all $(t,j) \in E$, $P_c(t,j) \geq 0$ and $P_d(t,j) \geq 0$, a lower bound on $P$ in \eqref{eqn:PEP} is $P(t,j) \geq q_m I$ for all $(t,j) \in E$.
Next, we develop an upper bound on $P$. 
Since, by Assumptions~\ref{asm:COMABStructure} and \ref{asm:COMHybridPE}, the conditions of Theorem \ref{thm:PEMain} are satisfied, it follows from \eqref{eqn:PEExpStabilityA} and from the equivalence between $|\xi|_{\cal A}$ and $|\vartheta|$ that, for each solution $\xi = (\vartheta,\tau,k)$ to $\HS_0$ and each $(s,i),(t,j) \in \dom \xi$ satisfying $s \geq t$ and $i \geq j$, 
%
    $|\vartheta(s,i)| \leq \kappa_0 \mathrm{e}^{-\lambda_0(s + i - t - j)} |\vartheta(t,j)|$
%
with $\kappa_0$ and $\lambda_0$ from Theorem \ref{thm:PEMain}.
By substituting \eqref{eqn:PEPhiStateTransition} into the expression above, we have that, for each $(s,i),(t,j) \in \dom \xi$ satisfying $s \geq t$ and $i \geq j$,
%
    $|\Phi(s,i,t,j) \vartheta(t,j)|  
    \leq \kappa_0 \mathrm{e}^{-\lambda_0(s + i - t - j)} |\vartheta(t,j)|$
%
which, if $|\vartheta(t,j)| \neq 0$, implies that
%
    $|\Phi(s,i,t,j) \vartheta(t,j)| / |\vartheta(t,j)|  
    \leq \kappa_0 \mathrm{e}^{-\lambda_0(s + i - t - j)}$. 
%
Since this inequality holds for any $\vartheta(t,j) \in \reals^p \setminus \{0\}$, it follows from the equivalence between $\dom \xi$ and $E$ that, for each $(s,i),(t,j) \in E$ satisfying $s \geq t$ and $i \geq j$,
%
    $|\Phi(s,i,t,j)|
    = \sup_{r \in \reals^p \setminus \{0\}} 
    |\Phi(s,i,t,j) r|/|r|
    \leq \kappa_0 \mathrm{e}^{-\lambda_0(s + i - t - j)}$. 
%
Then, from the definitions of $P_c$ and $P_d$ below \eqref{eqn:PEP},
%
    $P_c(t,j) 
    %
    %
    \leq q_M \int_{0}^{\infty} \kappa_0^2 \mathrm{e}^{-2 \lambda_0 s} d s I
    = \frac{q_M \kappa_0^2}{2 \lambda_0} I$ 
%
and
%
    $P_d(t,j) 
    \leq q_M \sum_{i = 0}^{\infty} \kappa_0^2 \mathrm{e}^{- 2 \lambda_0 i} I
    = \frac{q_M \kappa_0^2 \mathrm{e}^{2 \lambda_0}}{\mathrm{e}^{2 \lambda_0} - 1} I$. 
%
%
From the bounds above and the definition of $P$ in \eqref{eqn:PEP}, we conclude that \eqref{eqn:PEPbound} holds with $p_m$, $p_M$ \iftoggle{arxiv}{in \eqref{eqn:PEpm_pM}}{as in Lemma~\ref{lem:PEP}}.

Next, we show that \eqref{eqn:PEPdot} holds. We differentiate $P$ during flows and use that, for each $(s,i) \in E$ and each $j \in \nats$ and for almost all $t \in I^j$,
%
    $\frac{d}{dt} \Phi(s,i,t,j) 
    = \Phi(s,i,t,j) A(t,j)$. 
%
This property follows from \eqref{eqn:PEUdynamicsFlow} and from the definition of $\Phi$ in \eqref{eqn:PEPhiDefn}. 
Using the Leibniz integral rule, we obtain that, for each $j \in \nats$ and for almost all $t \in I^j$,
%
    $\frac{d}{dt} P_c(t,j)
    = A(t,j)^\top P_c(t,j) + P_c(t,j) A(t,j) - Q(t,j)$ 
%
and 
%
    $ \frac{d}{dt} P_d(t,j)
    = A(t,j)^\top P_d(t,j) + P_d(t,j) A(t,j)$. 
%
Combining these expressions and using \eqref{eqn:PEP} and item 1 of Assumption~\ref{asm:COMABStructure}, we conclude that \eqref{eqn:PEPdot} holds.

To conclude the proof, we show that \eqref{eqn:PEPplus} holds. We use that, for each $(t,j),(s,i) \in \Upsilon(E)$,
%
    $\Phi(s,i,t,j+1)(I - B(t,j))
    = \Phi(s,i,t,j)$. 
%
This property follows from \eqref{eqn:PEUdynamicsJump} and the definition of $\Phi$ in \eqref{eqn:PEPhiDefn}.
Then, for each $(t,j) \in \Upsilon(E)$,
%
%
    $(I - B(t,j))^\top P_c(t,j+1) (I - B(t,j))
    = \sum_{i = j+1}^{J} \int_{\max\{t, t_i\}}^{t_{i+1}} \Phi^\top(s,i,t,j) Q(s,i) \Phi(s,i,t,j) d s$. 
%
Since the value of ordinary time $t$ is the same immediately before and after each jump, it follows that, for each $(t,j) \in \Upsilon(E)$, $\max\{t, t_j\} = t = t_{j+1}$. Hence,
%
    $(I - B(t,j))^\top P_c(t,j+1) (I - B(t,j))
    %
    %
    %
    = P_c(t,j)$. 
%
Next, for each $(t,j) \in \Upsilon(E)$, since $t = t_{j+1}$ at each jump, 
%
    $(I - B(t,j))^\top P_d(t,j+1) (I - B(t,j)) - P_d(t,j)
    = - \Phi^\top(t,j,t,j) Q(t,j) \Phi(t,j,t,j)
    = - Q(t,j)$. 
%
Combining these expressions and using \eqref{eqn:PEP} and item 2 of Assumption~\ref{asm:COMABStructure}, we conclude that \eqref{eqn:PEPplus} holds.
\end{proof}
}
We now have all the ingredients to prove Theorem \ref{thm:COMISS}.
\iftoggle{arxiv}{
\begin{proof}[\unskip\nopunct]{\bf Proof of Theorem \ref{thm:COMISS}:}
Since, by Assumptions~\ref{asm:COMABStructure} and \ref{asm:COMHybridPE}, the conditions of Lemma \ref{lem:PEP} are satisfied, given $q_M \geq q_m > 0$ and a symmetric matrix function $Q : E \to \reals^{p \times p}$ satisfying \eqref{eqn:PEQbound}, there exists a symmetric matrix function $P : E \to \reals^{p \times p}$ satisfying \eqref{eqn:PEPbound}--\eqref{eqn:PEPplus}. Given such $P$, consider the Lyapunov function
\begin{align*}
    V(\xi) := \vartheta^\top P(\tau,k) \vartheta
    \qquad \forall \xi \in C \cup D.
\end{align*}
From \eqref{eqn:PEPbound} and from the equivalence between $|\vartheta|$ and $|\xi|_{\cal A}$, we have that
\begin{align} \label{eqn:PEVbound}
    p_m |\xi|^2_{\cal A} \leq V(\xi) \leq p_M |\xi|^2_{\cal A}
    \quad \forall \xi \in C \cup D,
\end{align}
with $p_m, p_M$ as in \eqref{eqn:PEpm_pM}. We first study the change in $V$ during flows. Omitting the $(\tau,k)$ arguments for readability, we have from \eqref{eqn:PEPdot} that, for all $\xi \in C$,
\begin{equation*}
\begin{aligned}
    \langle \nabla V(\xi), F(\xi) \rangle
    \leq - \vartheta^\top Q \vartheta + 2 \vartheta^\top P d_c.
\end{aligned}
\end{equation*}
We use that for any $\varrho > 0$, $2 \vartheta^\top P d_c \leq \varrho \vartheta^\top P \vartheta + \varrho^{-1} d_c^\top P d_c$. Choosing $\varrho = q_m / (2 p_M)$, 
\begin{align*}
    \langle \nabla V(\xi), F(\xi) \rangle
    &\leq - \frac{q_m}{2 p_M} V(\xi) + \frac{2 p_M^2}{q_m} |d_c(\tau,k)|^2.
\end{align*}
Let $\zeta \in (0, 1)$. By adding and subtracting $\zeta \frac{q_m}{2 p_M} V(\xi)$ to the right-hand side of the expression above, we conclude
%
%
\begin{align} \label{eqn:PEISSLyapflow}
    &\langle \nabla V(\xi), F(\xi) \rangle
    \leq - \frac{q_m}{2 p_M} (1 - \zeta) V(\xi) \\
    &\qquad \forall \xi \in C : \ V(\xi) \geq \frac{4 p_M^3}{q_m^2 \zeta} |d_c(\tau,k)|^2. \notag
\end{align}
Next, we study the change in $V$ at jumps. For readability, we omit the $(\tau,k)$ arguments and denote $P(\tau,k+1)$ as $P^+$. We have from \eqref{eqn:PEPplus} that, for all $\xi \in D$,
\begin{align*}
    V(G(\xi)) - V(\xi)
    &\leq - \vartheta^\top Q \vartheta + 2 |\vartheta^\top P^+ d_d| + d_d^\top P^+ d_d.
\end{align*}
We use that for any $\varrho > 0$, $2 |\vartheta^\top P^+ d_d| \leq \varrho \vartheta^\top P^+ \vartheta + \varrho^{-1} d_d^\top P^+ d_d$. Choosing $\varrho = q_m / (2 p_M)$,
\begin{equation*}
\mbox{$\squeezespaces{0.4}
\begin{aligned}
    V(G(\xi)) - V(\xi)
    &\leq - \frac{q_m}{2 p_M} V(\xi) + \left( \frac{2 p_M^2}{q_m} + p_M \right) |d_d(\tau,k)|^2.
\end{aligned}
$}
\end{equation*}
Let $\zeta \in (0, 1)$. By adding and subtracting $\zeta \frac{q_m}{2 p_M} V(\xi)$ to the right-hand side of the expression above, we conclude 
%
%
\begin{align} \label{eqn:PEISSLyapjump}
    &V(G(\xi)) - V(\xi)
    \leq - \frac{q_m}{2 p_M} (1 - \zeta) V (\xi), \\
    &\qquad \forall \xi \in D : \ V(\xi) \geq \frac{2 p_M^2}{q_m \zeta} \left( \frac{2 p_M}{q_m} + 1 \right) |d_d(\tau,k)|^2. \notag
\end{align}
Note that the lower bound on $V$ for which \eqref{eqn:PEISSLyapjump} holds is more restrictive than the lower bound on $V$ for which \eqref{eqn:PEISSLyapflow} holds.
Using the function $d$ defined in \eqref{eqn:COMd}, we combine the expressions in \eqref{eqn:PEISSLyapflow} and \eqref{eqn:PEISSLyapjump} and obtain
\begin{equation*}
\begin{aligned}
    \langle \nabla V(\xi), F(\xi) \rangle &\leq \! - \frac{q_m}{2 p_M} (1 - \zeta) V(\xi) &
    &\forall \xi \in C \cap S \\
    V(G(\xi)) - V(\xi) &\leq - \frac{q_m}{2 p_M} (1 - \zeta) V(\xi) & 
    &\forall \xi \in D \cap S,
\end{aligned}
\end{equation*}
where
%
    $\mbox{$\squeezespaces{0.5}
    S := \big\{\xi \in C \cup D : V(\xi) \geq \frac{2 p_M^2}{q_m \zeta} \left( \frac{2 p_M}{q_m} + 1 \right) |d(\tau,k)|^2 \big\}.
    $}$
%
Then, for each solution $\xi$ to $\HS$, by integration using the bounds above, we have that, for all $(t,j) \in \dom \xi$,
\begin{equation*}
\begin{aligned}
    &\mbox{\scriptsize$\displaystyle
    V(\xi(t,j)) \leq \exp \bigg\{ \bigg( 
    - \frac{q_m}{2 p_M} (1 - \zeta) t
    + \ln \left( 1 - \frac{q_m}{2 p_M} (1 - \zeta) \right) j
    \bigg) \bigg\} $} \\
    &\qquad\qquad\quad \mbox{\scriptsize$\displaystyle
    \times V(\xi(0,0)) 
    + \frac{2 p_M^2}{q_m \zeta} \left( \frac{2 p_M}{q_m} + 1 \right)
    \|d\|_{(t,j)}. $}
\end{aligned}
\end{equation*}
Using \eqref{eqn:PEVbound}, we conclude that \eqref{eqn:COMISS} holds.
\end{proof}
}{
\begin{proof}[\unskip\nopunct]{\bf Proof of Theorem \ref{thm:COMISS}:}
Since, by Assumptions~\ref{asm:COMABStructure} and \ref{asm:COMHybridPE}, the conditions of Lemma \ref{lem:PEP} are satisfied, given $q_M \geq q_m > 0$ and a symmetric matrix function $Q : E \to \reals^{p \times p}$ satisfying \iftoggle{arxiv}{\eqref{eqn:PEQbound}}{$q_m I \leq Q(t,j) \leq q_M I$ for all $(t,j) \in E$}, there exists a symmetric matrix function $P : E \to \reals^{p \times p}$ satisfying \eqref{eqn:PEPbound}--\eqref{eqn:PEPplus}. Given such $P$, consider the Lyapunov function
\begin{align*}
    V(\xi) := \vartheta^\top P(\tau,k) \vartheta
    \qquad \forall \xi \in C \cup D.
\end{align*}
From \eqref{eqn:PEPbound} and the equivalence between $|\vartheta|$ and $|\xi|_{\cal A}$,
\begin{align} \label{eqn:PEVbound}
    p_m |\xi|^2_{\cal A} \leq V(\xi) \leq p_M |\xi|^2_{\cal A}
    \quad \forall \xi \in C \cup D,
\end{align}
with $p_m, p_M$ \iftoggle{arxiv}{as in \eqref{eqn:PEpm_pM}}{as in Lemma~\ref{lem:PEP}}. We first study the change in $V$ during flows. Omitting the $(\tau,k)$ arguments for readability, we have from \eqref{eqn:PEPdot} that, for all $\xi \in C$,
%
    $\langle \nabla V(\xi), F(\xi) \rangle
    \leq - \vartheta^\top Q \vartheta + 2 \vartheta^\top P d_c$. 
%
We use that for any $\varrho > 0$, $2 \vartheta^\top P d_c \leq \varrho \vartheta^\top P \vartheta + \varrho^{-1} d_c^\top P d_c$. Choosing $\varrho = q_m / (2 p_M)$, for each $\zeta \in (0, 1)$,
%
%
\begin{equation} \label{eqn:PEISSLyapflow}
\begin{aligned} 
    &\langle \nabla V(\xi), F(\xi) \rangle
    \leq - \frac{q_m}{2 p_M} (1 - \zeta) V(\xi) \\
    &\qquad \forall \xi \in C : \ V(\xi) \geq \frac{4 p_M^3}{q_m^2 \zeta} |d_c(\tau,k)|^2.
\end{aligned}
\end{equation}
Next, we study the change in $V$ at jumps. For readability, we omit the $(\tau,k)$ arguments and denote $P(\tau,k+1)$ as $P^+$. We have from \eqref{eqn:PEPplus} that, for all $\xi \in D$,
%
    $V(G(\xi)) - V(\xi)
    \leq - \vartheta^\top Q \vartheta + 2 |\vartheta^\top P^+ d_d| + d_d^\top P^+ d_d$. 
%
We use that for any $\varrho > 0$, $2 |\vartheta^\top P^+ d_d| \leq \varrho \vartheta^\top P^+ \vartheta + \varrho^{-1} d_d^\top P^+ d_d$. Choosing $\varrho = q_m / (2 p_M)$, for each $\zeta \in (0, 1)$,  
%
%
\begin{align} \label{eqn:PEISSLyapjump}
    &V(G(\xi)) - V(\xi)
    \leq - \frac{q_m}{2 p_M} (1 - \zeta) V (\xi), \\
    &\qquad \forall \xi \in D : \ V(\xi) \geq \frac{2 p_M^2}{q_m \zeta} \left( \frac{2 p_M}{q_m} + 1 \right) |d_d(\tau,k)|^2. \notag
\end{align}
Note that the lower bound on $V$ for which \eqref{eqn:PEISSLyapjump} holds is more restrictive than the lower bound on $V$ for which \eqref{eqn:PEISSLyapflow} holds.
Using the function $d$ defined in \eqref{eqn:COMd}, we combine the expressions in \eqref{eqn:PEISSLyapflow} and \eqref{eqn:PEISSLyapjump} and obtain
\begin{equation*}
\begin{aligned}
    \langle \nabla V(\xi), F(\xi) \rangle &\leq \! - \frac{q_m}{2 p_M} (1 - \zeta) V(\xi) &
    &\forall \xi \in C \cap S \\
    V(G(\xi)) - V(\xi) &\leq - \frac{q_m}{2 p_M} (1 - \zeta) V(\xi) & 
    &\forall \xi \in D \cap S,
\end{aligned}
\end{equation*}
where
%
    $\mbox{$\squeezespaces{0.5}
    S := \big\{\xi \in C \cup D : V(\xi) \geq \frac{2 p_M^2}{q_m \zeta} \left( \frac{2 p_M}{q_m} + 1 \right) |d(\tau,k)|^2 \big\}.
    $}$
%
Then, for each solution $\xi$ to $\HS$, by integration using \eqref{eqn:PEVbound} and the bounds above, \eqref{eqn:COMISS} holds.
\end{proof}
}

\section{Proof of Theorem \ref{thm:COMISS2}} \label{apx:ISS2}
\begin{proof}[\unskip\nopunct]
Let $\xi$ be a maximal solution to $\HS$.
First, we upper bound $(t,j) \mapsto |\xi(t,j)|_{\cal A}$ for all $(t,j) \in \dom \xi$. 
Since, by Assumptions~\ref{asm:COMABStructure} and \ref{asm:COMHybridPE}, the conditions of Theorem \ref{thm:COMISS} are satisfied, it follows from \eqref{eqn:COMISS} that
\begin{align} \label{eqn:PExM}
    |\xi(t,j)|_{\cal A} 
    &\leq \beta(|\xi(0,0)|_{\cal A},0) + a \rho |d(0,0)| =: \xi_M
\end{align}
for all $(t,j) \in \dom \xi$, where the second inequality follows from \eqref{eqn:PEdbound}.
Next, we define $\delta \mapsto c_1(\delta) \in \reals$ as
\begin{align} \label{eqn:PEc1}
    c_1(\delta) := -\frac{1}{b} \ln \left( \frac{\delta/2}{a \rho |d(0,0)|} \right) 
    \quad \forall \delta > 0.
\end{align}
Let $\delta > 0$ be such that there exists $(t',j') \in \dom \xi$ such that $t' + j' \geq c_1(\delta)$. Then, it follows from \eqref{eqn:PEdbound} that for all $(t,j) \in \dom \xi$ satisfying $t \geq t'$ and $j \geq j'$,
%
    $|d(t,j)|
    \leq a \mathrm{e}^{-b(t+j)} |d(0,0)| 
    \leq a \mathrm{e}^{-b c_1(\delta)} |d(0,0)| 
    = \delta/(2\rho).$
%
Hence, for all $(t,j) \in \dom \xi$ satisfying $t \geq t'$ and $j \geq j'$, the supremum norm of $(t,j) \mapsto |d(t,j)|$ from $(t',j')$ to $(t,j)$ is less than or equal to $\delta/(2\rho)$. Thus, from \eqref{eqn:COMISS}, for all $(t,j) \in \dom \xi$ satisfying $t \geq t'$ and $j \geq j'$,
\begin{align} \label{eqn:PErhobound}
    |\xi(t,j)|_{\cal A} 
    &\leq \beta(\xi_M,t+j-c_1(\delta)) + \delta/2
\end{align}
with $\xi_M$ as in \eqref{eqn:PExM}. Next, we define $\delta \mapsto c_2(\delta) \in \reals$ as
\begin{align} \label{eqn:PEc2}
    c_2(\delta) := - \frac{1}{\omega} \ln \left( \frac{\delta/2}{ \beta(\xi_M,0) } \right)
    \quad \forall \delta > 0.
\end{align}
Omitting the argument $\delta$ of $c_1$ and $c_2$ for readability, we have that, for all $(t,j) \in \dom \xi$ satisfying $t + j \geq c_2 + c_1$,
\begin{align} \label{eqn:PEbetabound}
    \beta(\xi_M,t+j-c_1) 
    \leq \beta(\xi_M,c_2 + c_1 -c_1) 
    = \delta/2.
\end{align}
By combining \eqref{eqn:PErhobound} and \eqref{eqn:PEbetabound}, it follows that, for each $\delta > 0$ and each $(t,j) \in \dom \xi$,
\begin{equation} \label{eqn:PExdelta}
\hspace{-3mm}
\mbox{$\squeezespaces{0.5}
\begin{aligned}
t + j \geq \max\{ c_1(\delta), c_2(\delta) + c_1(\delta) \} 
\implies |\xi(t,j)|_{\cal A} \leq \delta.
\end{aligned}
$}
\hspace{-3mm}
\end{equation}
%
%
Since $c_1$ in \eqref{eqn:PEc1} and $c_2$ in \eqref{eqn:PEc2} are continuous monotonically increasing functions of $\delta$ with $\rge c_1 = \rge c_2 = \reals$, it follows that, for each $(t,j) \in \dom \xi$, there exists a unique $\delta > 0$ such that $t + j = \max\{ c_1(\delta), c_2(\delta) + c_1(\delta) \}$. For such $\delta$, \eqref{eqn:PExdelta} holds. Hence, we develop a bound for $(t,j) \mapsto |\xi(t,j)|_{\cal A}$ by bounding, for each $(t,j) \in \dom \xi$, the corresponding value of $\delta$ for which $t + j = \max\{ c_1(\delta), c_2(\delta) + c_1(\delta) \}$. 

Given $(t,j) \in \dom \xi$ and $\delta > 0$ satisfying $t + j = \max\{ c_1(\delta), c_2(\delta) + c_1(\delta) \}$, we consider two cases: $\max\{ c_1(\delta), c_2(\delta) + c_1(\delta) \} = c_1(\delta)$ and $\max\{ c_1(\delta), c_2(\delta) + c_1(\delta) \} = c_2(\delta) + c_1(\delta)$.
\begin{enumerate}[label=\arabic*.]
\item If $\max\{ c_1(\delta), c_2(\delta) + c_1(\delta) \} = c_1(\delta)$, then $t + j = c_1(\delta)$
which, from \eqref{eqn:PEc1}, implies that
%
    $\delta = 2 a \rho \mathrm{e}^{-b(t+j)} |d(0,0)|$.
%
\item If $\max\{ c_1(\delta), c_2(\delta) + c_1(\delta) \} = c_2(\delta) + c_1(\delta)$, then $t + j = c_2(\delta) + c_1(\delta)$.
Since $c_2(\delta) + c_1(\delta) \geq c_1(\delta)$, it follows that $c_2(\delta) \geq 0$. Then, we consider two cases: $c_1(\delta) \leq 0$ and $c_1(\delta) > 0$.
\begin{itemize}
    \item[a.] If $c_1(\delta) \leq 0$, then $t + j \leq c_2(\delta)$ which, from \eqref{eqn:PEc2}, implies that
    %
    %
        $\delta \leq 2 \mathrm{e}^{-\omega(t+j)} \beta(\xi_M,0)$.
    %
    Substituting $\xi_M$ given in \eqref{eqn:PExM} yields\linebreak
    $\delta~\leq~\max \left\{ 2 \frac{p_M}{p_m}, 2 a \rho \sqrt{\frac{p_M}{p_m}} \right\} \mathrm{e}^{-\omega(t + j)} ( |\xi(0,0)|_{\cal A} + |d(0,0)| ).$
    %
    \item[b.] If $c_1(\delta) > 0$, we define $\sigma := \min\{\omega, b\}$ and then
    %
    \begin{align*}
        t + j 
        &= c_1(\delta) + c_2(\delta) \\
        %
        &\leq
        \mbox{$\squeezespaces{0.3}
        - \dfrac{1}{\sigma} 
        \bigg(
        \ln \left( \dfrac{\delta/2}{ \beta(\xi_M,0) } \right) 
        + \ln \left( \dfrac{\delta/2}{a \rho |d(0,0)|} \right)
        \bigg) $}
    \end{align*}
    which implies 
    \iftoggle{long}{that
    \begin{align*}
        \delta \leq 
        \big( 4 a \rho \mathrm{e}^{-\sigma(t + j)} \beta(\xi_M,0) |d(0,0)| \big)^{\sfrac{1}{2}}.
    \end{align*}
    }{
        $\mbox{$\squeezespaces{0.1}
        \delta \leq 
        \sqrt{ 4 a \rho \mathrm{e}^{-\sigma(t + j)} \beta(\xi_M,0) |d(0,0)| \big) }
        $}$. 
    }
    By substituting $\xi_M$ in \eqref{eqn:PExM} and completing the square yields
    $\mbox{$\squeezespaces{0.7} \delta \leq \max \left\{ \left(\frac{p_M}{p_m}\right)^{\sfrac{3}{4}},
    2 a \sqrt{\rho} \left(\frac{p_M}{p_m}\right)^{\sfrac{1}{4}} \right\}$} \allowbreak \times \mathrm{e}^{-\frac{\sigma}{2}(t + j)}
    \left( |\xi(0,0)|_{\cal A} + |d(0,0)| \right).$
    %
    %
\end{itemize}
\end{enumerate}
%
%
\iftoggle{arxiv}{By combining the bounds in the items above, and using that}{Using the bounds above and that} $p_M/p_m > 1$ and $\rho > 1$, it follows from \eqref{eqn:PExdelta} that \eqref{eqn:PExibound} holds.
\end{proof}

\iftoggle{arxiv}{
\section{Proof of Lemma \ref{prop:PEepsStab}} \label{apx:epsStab}
\begin{proof}[\unskip\nopunct]
Consider the Lyapunov function
\begin{align} \label{eqn:PEVvarepsilon}
    V_\varepsilon(\xi) 
    &:= \frac{1}{2} \varepsilon^\top \varepsilon
    \quad \forall \xi \in \widetilde{C}_g \cup \widetilde{D}_g,
\end{align}
with $\varepsilon$ as in \eqref{eqn:PEvarpeilonHSg}.
Since $\theta$ is constant, we have from \eqref{eqn:PEHSgerr} that
%
    $\dot \varepsilon 
    = \dot z + \dot \eta - \dot \psi \theta
    = -\lambda_c \varepsilon$.
%
Thus, for all $\xi \in \widetilde{C}_g$,
\iftoggle{long}{
\begin{align*} 
    \langle \nabla V_\varepsilon(\xi), \widetilde{F}_g(\xi) \rangle 
    = - 2 \lambda_c V_\varepsilon(\xi)
    \leq 0.
\end{align*}
}{
    $\langle \nabla V_\varepsilon(\xi), \widetilde{F}_g(\xi) \rangle 
    = - 2 \lambda_c V_\varepsilon(\xi)
    \leq 0.$ 
}
At jumps, since $\theta$ is constant, we have from \eqref{eqn:PEHSgerr} that
%
    $\varepsilon^+ 
    = z^+ + \eta^+ - \psi^+ \theta 
    = (1 - \lambda_d) \varepsilon$.
%
%
Thus, for all $\xi \in \widetilde{D}_g$,
\iftoggle{long}{
\begin{align*} 
    V_\varepsilon(\widetilde{G}_g(\xi)) - V_\varepsilon(\xi)
    = - \lambda_d ( 2 - \lambda_d ) V_\varepsilon(\xi)
    \leq 0
\end{align*}
}{
    $V_\varepsilon(\widetilde{G}_g(\xi)) - V_\varepsilon(\xi)
    = - \lambda_d ( 2 - \lambda_d ) V_\varepsilon(\xi)
    \leq 0$,
}
where the inequality holds since $\lambda_d \in (0, 2)$. 
Then, for each solution $\xi$ to $\widetilde{\HS}_g$, by integration using the bounds above and the definition of $V_\varepsilon$ in \eqref{eqn:PEVvarepsilon}, \eqref{eqn:PEepsbound} holds.
\end{proof}
}{}
\iftoggle{arxiv}{
\section{Proof of Lemma \ref{prop:PEpsiStab}} \label{apx:psiStab}
\begin{proof}[\unskip\nopunct]
Consider the Lyapunov function
\begin{align} \label{eqn:PEVpsi}
    V_\psi(\xi) 
    := \frac{1}{2} \tr (\psi^\top \psi) = \frac{1}{2} |\psi|_{\mathrm F}^2
    \quad \forall \xi \in \widetilde{C}_g \cup \widetilde{D}_g.
\end{align}
For all $\xi \in \widetilde{C}_g$, we have from \eqref{eqn:PEHSgerr} that
\iftoggle{long}{
\begin{align*} %
    \langle \nabla V_\psi(\xi), \widetilde{F}_g(\xi) \rangle
    &= - 2 \lambda_c V_\psi(\xi) + \tr(\psi^\top \phi_c(\tau,k)).
\end{align*}
}{
    $\langle \nabla V_\psi(\xi), \widetilde{F}_g(\xi) \rangle
    = - 2 \lambda_c V_\psi(\xi) + \tr(\psi^\top \phi_c(\tau,k))$. 
}
Applying the Cauchy-Schwarz inequality yields
\iftoggle{long}{
\begin{align*} %
    \langle \nabla V_\psi(\xi), \widetilde{F}_g(\xi) \rangle 
    &\leq - 2 \lambda_c V_\psi(\xi) + \sqrt{2 V_\psi(\xi)} |\phi_c(\tau,k)|_{\mathrm F}.
\end{align*}
}{
    $\langle \nabla V_\psi(\xi), \widetilde{F}_g(\xi) \rangle 
    \leq - 2 \lambda_c V_\psi(\xi) + \sqrt{2 V_\psi(\xi)} |\phi_c(\tau,k)|_{\mathrm F}$.
}
Hence,
\begin{align} \label{eqn:PEIDDVpsiflow}
    &\langle \nabla V_\psi(\xi), \widetilde{F}_g(\xi) \rangle 
    \leq 0 \\
    &\qquad \forall \xi \in \widetilde{C}_g : 
    V_\psi(\xi) \geq \frac{1}{2 \lambda_c^2} |\phi_c(\tau,k)|_{\mathrm F}^2. \notag
\end{align}
Let us now analyze the variation of $V_\psi$ at jumps.
Omitting the $(\tau,k)$ arguments for readability, for all $\xi \in \widetilde{D}_g$, we have from \eqref{eqn:PEHSgerr} that
%
    $V_\psi(\widetilde{G}_g(\xi)) - V_\psi(\xi)
    \leq - \bar{\lambda}_d V_\psi(\xi) 
    + | \tr (\psi^\top \phi_d) |
    + \frac{1}{2} \tr (\phi_d^\top \phi_d)$, 
%
where the inequality follows since $\lambda_d \in (0, 2)$, and we define $\bar{\lambda}_d := \lambda_d (2 - \lambda_d)$ for readability.
We apply the Cauchy-Schwarz inequality and use that for any $\varrho > 0$, $\sqrt{\tr (\psi^\top \psi)}  \sqrt{\tr (\phi_d^\top \phi_d)} \leq \varrho \tr (\psi^\top \psi) + \varrho^{-1} \tr (\phi_d^\top \phi_d)$. Choosing $\varrho = \bar{\lambda}_d / 4$ yields
%
    $V_\psi(\widetilde{G}_g(\xi)) - V_\psi(\xi)
    \leq - \frac{\bar{\lambda}_d}{2} V_\psi(\xi) 
    + \frac{\bar{\lambda}_d + 8}{2 \bar{\lambda}_d} |\phi_d|_{\mathrm F}^2$.
%
Hence,
\begin{equation} \label{eqn:PEVpsijump}
\begin{aligned} 
    &V_\psi(\widetilde{G}_g(\xi)) - V_\psi(\xi)
    \leq 0\\
    &\qquad \forall \xi \in \widetilde{D}_g : \ V_\psi(\xi) \geq \frac{\bar{\lambda}_d + 8}{\bar{\lambda}_d^2} |\phi_d(\tau,k)|_{\mathrm F}^2.
\end{aligned}
\end{equation}
Using the bounds in \iftoggle{dissertation}{Assumption~\ref{asm:PEHybridPEPsi}}{item~1 of Theorem~\ref{thm:PEGlobalStabilityHSg}}, we combine the expressions in \eqref{eqn:PEIDDVpsiflow} and \eqref{eqn:PEVpsijump} to obtain
\begin{align*}
    \langle \nabla V_\psi(\xi), \widetilde{F}_g(\xi) \rangle 
    &\leq 0
    \quad\forall \xi \in \widetilde{C}_g \cap S_\psi \\
    V_\psi(\widetilde{G}_g(\xi)) - V_\psi(\xi)
    &\leq 0 
    \quad\forall \xi \in \widetilde{D}_g \cap S_\psi,
\end{align*}
where
\iftoggle{dissertation}
{
\begin{equation*}
    S_\psi := \left\{\xi \in \widetilde{C}_g \cup \widetilde{D}_g : V_\psi(\xi) \geq \max \left\{
        \frac{1}{2 \lambda_c^2}, 
        \frac{\bar{\lambda}_d + 8}{\bar{\lambda}_d^2}
        \right\}
        \phi_M^2 \right\}.
\end{equation*}
}
{
\begin{align*}
    \mbox{\scriptsize$\displaystyle
    S_\psi := \left\{\xi \in \widetilde{C}_g \cup \widetilde{D}_g : V_\psi(\xi) \geq \max \left\{
        \frac{1}{2 \lambda_c^2}, 
        \frac{\bar{\lambda}_d + 8}{\bar{\lambda}_d^2}
     \right\}
     \phi_M^2 \right\}.
    $}
\end{align*}
}
%
%
%
Then, for each solution $\xi$ to $\widetilde{\HS}_g$ from ${\cal X}_0$, by integration using the bounds above, we conclude that, for all $(t,j) \in \dom \xi$, 
%
    $V_\psi(\xi(t,j))  
    \leq V_\psi(\xi(0,0)) + \max \left\{
        1/(2 \lambda_c^2), 
        (\bar{\lambda}_d + 8)/\bar{\lambda}_d^2
     \right\}
     \phi_M^2$.
%
Using the definition of $V_\psi$ in \eqref{eqn:PEVpsi}, we obtain
%
    $|\psi(t,j)| \leq 
    |\psi(t,j)|_{\mathrm F} 
    \leq |\psi(0,0)|_{\mathrm F} 
    + \max \left\{
        1/\lambda_c, 
        \sqrt{2 \bar{\lambda}_d + 16}/\bar{\lambda}_d
     \right\}
     \phi_M
    \leq \psi_0 
    + \max \left\{
        1/\lambda_c, 
        \sqrt{2 \bar{\lambda}_d + 16}/\bar{\lambda}_d
     \right\}
     \phi_M$
%
for all $(t,j) \in \dom \xi$, where the last inequality follows from the fact that, since $\xi(0,0) \in {\cal X}_0$, $|\psi(0,0)|_{\mathrm F} \leq \psi_0$. Hence, \eqref{eqn:PEpsibound} holds.
\end{proof}
}{}
\iftoggle{arxiv}{
\section{Proof of Lemma \ref{lem:PEepsISS}} \label{apx:epsISS}
\begin{proof}[\unskip\nopunct]
Consider the Lyapunov function
\begin{align} \label{eqn:PEVepsilon}
    V_\varepsilon(\xi) 
    &:= \frac{1}{2} \varepsilon^\top \varepsilon
    \quad \forall \xi \in {C}_\nu \cup {D}_\nu,
\end{align}
with $\varepsilon$ as in \eqref{eqn:PEvarpeilonHSg}. 
Since $\theta$ is constant, we have from \eqref{eqn:PEvarpeilonHSg} that
%
    $\dot \varepsilon 
    = \dot z + \dot \eta - \dot \psi \theta
    = -\lambda_c \varepsilon + \alpha_c(\tau,k)$,
%
with $\alpha_c$ as in \eqref{eqn:PEalpha1}.
Thus, for all $\xi \in {C}_\nu$,
$\langle \nabla V_\varepsilon(\xi), {F}_\nu(\xi) \rangle
    = - 2 \lambda_c V_\varepsilon(\xi) - \varepsilon^\top \alpha_c(\tau,k)
    \leq - \lambda_c V_\varepsilon(\xi) + \frac{2}{\lambda_c} |\alpha_c(\tau,k)|^2$.
%
%
Then, for any $\zeta \in (0, 1)$,
%
%
\begin{align} \label{eqn:PEetaISSflow}
    &\langle \nabla V_\varepsilon(\xi), {F}_\nu(\xi) \rangle
    \leq - \lambda_c (1 - \zeta) V_\varepsilon(\xi) \\
    &\qquad \forall \xi \in {C}_\nu : \ V_\varepsilon(\xi) \geq \frac{2}{\lambda_c^2 \zeta} |\alpha_c(\tau,k)|^2. \notag
\end{align}
Let us now analyze the variation of $V_\varepsilon$ at jumps.
Omitting the $(\tau,k)$ arguments for readability, since $\theta$ is constant, we have from \eqref{eqn:PEvarpeilonHSg} that
%
    $\varepsilon^+ 
    = z^+ + \eta^+ - \psi^+ \theta 
    = (1 - \lambda_d) \varepsilon + \alpha_d(\tau,k)$,
%
with $\alpha_d$ as in \eqref{eqn:PEalpha2}.
Thus, for all $\xi \in {D}_\nu$,
    $V_\varepsilon({G}_\nu(\xi)) - V_\varepsilon(\xi)
    \leq - \bar{\lambda}_d V_\varepsilon(\xi) + |\varepsilon^\top \alpha_d| + \frac{1}{2} \alpha_d^\top \alpha_d
    \leq - \frac{\bar{\lambda}_d}{2} V_\varepsilon(\xi) 
    + \frac{\bar{\lambda}_d + 8}{2 \bar{\lambda}_d} |\alpha_d(\tau,k)|^2$,
where we define $\bar{\lambda}_d := \lambda_d (2 - \lambda_d)$ for readability.
Then, for any $\zeta \in (0, 1)$,
%
%
\begin{equation} \label{eqn:PEetaISSjump}
\begin{aligned} 
    &V_\varepsilon({G}_\nu(\xi)) - V_\varepsilon(\xi)
    \leq - \frac{\bar{\lambda}_d}{2} (1 - \zeta) V_\varepsilon(\xi), \\
    &\qquad \forall \xi \in {D}_\nu : \ V_\varepsilon(\xi) \geq \frac{\bar{\lambda}_d + 8}{\bar{\lambda}_d^2 \zeta} |\alpha_d(\tau,k)|^2.
\end{aligned}
\end{equation}
Using the function $d_\varepsilon$ defined in \eqref{eqn:PEdvarepsilon}, we combine the expressions in \eqref{eqn:PEetaISSflow} and \eqref{eqn:PEetaISSjump} and obtain that
\begin{align*}
    \langle \nabla V_\varepsilon(\xi), {F}_\nu(\xi) \rangle 
    &\leq - \lambda_c (1 - \zeta) V_\varepsilon(\xi)
    \quad\forall \xi \in {C}_\nu \cap S_\varepsilon \\
    V_\varepsilon({G}_\nu(\xi)) - V_\varepsilon(\xi)
    &\leq - \frac{\bar{\lambda}_d}{2} (1 - \zeta) V_\varepsilon(\xi)
    \quad\forall \xi \in {D}_\nu \cap S_\varepsilon,
\end{align*}
where
\iftoggle{dissertation}
{
\begin{equation*}
    S_\varepsilon := \left\{\xi \in {C}_\nu \cup {D}_\nu : V_\varepsilon(\xi) \geq \max \left\{
        \frac{2}{\lambda_c^2 \zeta}, 
        \frac{\bar{\lambda}_d + 8}{\bar{\lambda}_d^2 \zeta}
        \right\}
        |d_\varepsilon(\tau,k)|^2 \right\}.
\end{equation*}
}
{
\begin{align*}
    \mbox{\scriptsize$\displaystyle
    S_\varepsilon := \left\{\xi \in {C}_\nu \cup {D}_\nu : V_\varepsilon(\xi) \geq \max \left\{
        \frac{2}{\lambda_c^2 \zeta}, 
        \frac{\bar{\lambda}_d + 8}{\bar{\lambda}_d^2 \zeta}
     \right\}
     |d_\varepsilon(\tau,k)|^2 \right\}.
    $}
\end{align*}
}
\noindent Then, for each solution $\xi$ to ${\HS}_\nu$, by integration using the bounds above, and using the definition of $V_\varepsilon$ in \eqref{eqn:PEVepsilon}, we conclude that \eqref{eqn:PEepsISS} holds.
\end{proof}
}{}
\end{document}